\documentclass[11pt,reqno]{amsart}
\usepackage{xifthen,setspace}
\usepackage{macros}
\usepackage{colortbl}
\usepackage{xcolor}
\usepackage{tikz}

\allowdisplaybreaks

\usepackage[comma,sort&compress, numbers]{natbib}

\begin{document}

\thispagestyle{empty}
\title[Monochromatic Hyperedges in Multiplex Hypergraphs]{Joint Poisson Convergence of Monochromatic Hyperedges in Multiplex Hypergraphs}

\author{Yangxinyu Xie}
\address{Department of Statistics and Data Science, University of Pennsylvania, Philadelphia, USA } 
\email{xinyux@wharton.upenn.edu}

\author{Bhaswar B. Bhattacharya}
\address{Department of Statistics and Data Science, University of Pennsylvania, Philadelphia, USA }
\email{bhaswar@wharton.upenn.edu}

\subjclass[2010]{05C15, 60C05,  60F05}
\keywords{Birthday paradox, Combinatorial probability, Hypergraph coloring, Poisson approximation.}

\begin{abstract} 
Given a sequence of $r$-uniform hypergraphs $H_n$, denote by $T(H_n)$ the number of monochromatic hyperedges when the vertices of $H_n$ are colored uniformly at random with $c = c_n$ colors. In this paper, we study the joint distribution of monochromatic hyperedges for hypergraphs with multiple layers (multiplex hypergraphs). Specifically, we consider the joint distribution of $\bm T _n:= (T(H_n^{(1)}), T(H_n^{(2)}))$, for two sequences of hypergraphs $H_n^{(1)}$ and $H_n^{(2)}$ on the same set of vertices. We will show that the joint distribution of $\bm T_n $ converges to (possibly dependent) Poisson distributions whenever the mean vector and the covariance matrix of $\bm T_n$ converge. In other words, the joint Poisson approximation of $\bm T_n$ is determined only by the convergence of its first two moments. 
    This generalizes recent results on the second moment phenomenon for Poisson approximation from graph coloring to hypergraph coloring and from marginal convergence to joint convergence. Applications include generalizations of the birthday problem, counting monochromatic subgraphs in randomly colored graphs, and counting monochromatic arithmetic progressions in randomly colored integers.  Extensions to random hypergraphs and weighted hypergraphs are also discussed. 
\end{abstract}

\maketitle

\section{Introduction}
\label{sec: Introduction}

Fix $r \geq 2$ and consider a sequence of $r$-uniform hypergraphs $H_n= (V(H_n), E(H_n))$, with vertex set $V(H_n)$ and hyperedge set $E(H_n)$ such that $|E(H_n)| \rightarrow \infty$,   as $n \rightarrow \infty$. Suppose the vertices of $H_n$ are colored uniformly at random with $c = c_n$ colors, that is, 
\begin{align}\label{eq:cn}
\bP(v\in V(H_n) \text{ has color } a \in \{1, 2, \ldots, c_n \})=\frac{1}{c_n} , 
\end{align}
independently from the other vertices. Given such a coloring a hyperedge $\bm e \in E(H_n)$ is said to be {\it monochromatic} if all the vertices in $\bm e$ have the same color. Formally, if $X_v$ denotes the color of the vertex $v \in V(H_n)$, then an edge $\bm e \in E(H_n)$ is monochromatic if $X_{v} = X_{v'}$, for all $v, v' \in \bm e$. 
Denote by $T(H_n)$ the number of monochromatic hyperedges in $H_n$, that is, 
\begin{equation}
    \label{eq:THn}
    T(H_n) := \sum_{\bm e \in E(H_n)} \bm 1\{X_{=\bm e}\} , 
\end{equation} 
where $\bm 1\{X_{=\bm e}\} := \bm 1\{X_{v} = X_{v'} \text{ for all } v,v' \in \bm e\}$. This general framework includes numerous celebrated problems in combinatorial probability. The following are three examples:

\begin{example}[Birthday Problem] 
\label{example:monochromaticedgesgraph} 
When $r=2$, then $H_n$ is a graph and $T(H_n)$ counts the number of monochromatic edges in a uniformly random coloring of the vertices of $H_n$. Specifically, if $H_n$ is a friendship-graph (two people are connected by an edge in the graph if they are friends) colored uniformly randomly with $c=365$ colors (where the colors correspond to birthdays and the birthdays are assumed to be uniformly distributed across the year), then a monochromatic edge in $H_n$ corresponds to two friends with the same birthday. The birthday problem asks for the probability that there are two friends with the same birthday, that is, $\bP(T(H_n)>0)$. This is one of the most famous problems in elementary probability, which has been generalized in various directions and found applications in diverse fields (see \cite{barbour1992poisson,dasgupta2005matching,diaconis2002bayesian,diaconis1989methods,mitzenmacher2017probability,galbraith2012non,discretelogarithm,birthdayexchangeability} and the references therein). Note that $\bP(T(H_n)>0)=1-\bP(T(H_n)=0)=1-\chi_{H_n}(c)/c^{|V(H_n)|}$, where $\chi_{H_n}(c)$ counts the number of proper colorings of $H_n$ using $c$ colors. 
The function $\chi_{H_n}$ is known as the {\it chromatic polynomial} of $H_n$, a central object of study in graph theory \cite{dong2005chromatic,jensen2011graph,jukna2011extremal,mitzenmacher2017probability,van2001course}. A natural generalization of the birthday problem is to consider higher-order birthday matches; that is, the probability that there is a group of $r$ friends sharing the same birthday \cite{diaconis1989methods,nandi2007multicollision,generalizedbirthdayproblem,birthdaymulticollison}. 
This corresponds to counting the number of $r$-cliques in $H_n$ \cite{bhattacharya2020second}, which can be reformulated in terms of \eqref{eq:THn} as follows: Consider the $r$-uniform hypergraph $\cC_r(H_n) = (V(\cC_r(H_n)), E(\cC_r(H_n)))$, with $V(\cC_r(H_n)) = V(H_n)$ and 
$$E(\cC_r(H_n))  = \left\{S \in {V(H_n) \choose r} : S \text{ forms a } r\text{-clique in } H_n \right \},$$ 
where ${V(H_n) \choose r}$ denotes the collection of all $r$-element subsets of $V(H_n)$. Then $T(\cC_r(H_n))$ is precisely the number of monochromatic $r$-cliques (the number of $r$-th order birthday matches) in $H_n$. 
\end{example}

\begin{example}[Monochromatic subgraphs in random edge colorings]
\label{example:edgecoloring} 
Understanding the existence of monochromatic subgraphs in edge-colored graphs is one of the basic objectives in Ramsey theory \cite{graham1991ramsey,ramsey2015recent}. Specifically, Ramsey's theorem in its graph theoretic form, states that given a graph $F = (V(F), E(F))$, any $c$-coloring of the edges of a complete graph $K_n$ contains a monochromatic copy of $F$, whenever $n$ is large enough (depending on $c$ and $F$). Although Ramsey's theorem is a statement about any coloring, understanding what happens for a random coloring has important consequences. For instance, when $c=2$, for certain graphs $F$ (which are known as {\it common} graphs) the number of monochromatic copies of $F$ is asymptotically minimized by a random edge-coloring (see \cite{goodman1959sets,burr1980ramsey} and the references therein). Going beyond complete graphs to a general graph sequence $G_n = (V(G_n), E(G_n))$, with $|E(G_n)| \rightarrow \infty$, one can define $R(F, G_n)$ as the number of monochromatic copies of $F$ when the edges of $G_n$ are colored uniformly at random with $c$-colors.\footnote{In a uniformly random edge coloring of a graph $G_n = (V(G_n), E(G_n))$ with $c$ colors, each edge $e \in E(G_n)$ has color $a \in \{1, 2, \ldots, c\}$ with probability $1/c$, independent of the other edges. } This can be expressed in terms of \eqref{eq:THn} as follows: Consider the $|E(F)|$-uniform hypergraph $H_{G_n}(F) = (V(H_{G_n}(F)), E(H_{G_n}(F)) )$, where $V(H_{G_n}(F))  = E(G_n)$ and $E(H_{G_n}(F) )$ is the collection of $|E(F)|$-element subsets of $E(G_n)$ that form a copy of $F$ in $G_n$. Note that a uniformly random edge-coloring of the graph $G_n$ corresponds to a uniformly random vertex coloring of the hypergraph $H_{G_n}(F)$ , and hence, $T(H_{G_n}(F)) = R(F, G_n)$. 
\end{example}

\begin{example}[Monochromatic arithmetic progressions]  
\label{example:ap} 
Another key result in Ramsey theory is Van der Waerden's theorem \cite{van1927beweis,graham1991ramsey}, which states that given positive integers $c \geq 2$ and $r \geq 3$, any $c$-coloring of the integers $[n] :=\{1, 2, \ldots, n\}$ 
%(or an Abelian group $Z_n$ with $|Z_n| \rightarrow \infty$) 
contains a monochromatic $r$-term arithmetic progression ($r$-AP), that is, $r$ equally spaced integers of the same color, when $n$ is sufficiently large (depending on $r$ and $c$). As in  the graph theoretic setting, counting the number of monochromatic arithmetic progressions in a random coloring have important consequences (see  \cite{ap2017ramsey} and the references therein). Thus, for any set $A_n\subseteq [n]$, with $|A_n| \rightarrow \infty$, we can define $W_r(A_n)$ to be the number of monochromatic arithmetic progressions in a random $c$-coloring of the elements of $A_n$. To express $W_r(A_n)$ as \eqref{eq:THn} consider the hypergraph $H_{A_n}(r)$ with vertex set $A_n$ and edge set the collection of all $r$-term APs in $A_n$. Then $T(H_{A_n}(r))  = W_r(A_n)$. 
\end{example}

In this paper we will study the limiting distribution of $T(H_n)$ in the regime where both $|E(H_n)| \rightarrow \infty$ and $c= c_n \rightarrow \infty$ such that 
\begin{align}\label{eq:ETHn}
\bE T(H_n) = \frac{1}{c_n} |E(H_n)| \to \lambda \geq 0. 
\end{align}
This problem is well understood when $r = 2$, that is, when $H_n$ is a graph and $T(H_n)$ counts the number of monochromatic edges. In this case, $T(H_n)$ exhibits a {\it first moment phenomenon}: $$T(H_n)\dto \Pois(\lambda)$$ for any graph sequence $H_n$ for which \eqref{eq:ETHn} holds (see \cite[Theorem 5.G]{barbour1992poisson} and \cite[Theorem 1.1]{bhattacharya2017universal}). However, when $r \ge 3$, only assuming the convergence of the first moment as in \eqref{eq:ETHn} is not enough for Poisson convergence. For $r \geq 3$, one can construct a sequence of $r$-uniform hypergraphs $H_n$, for which $\bE T(H_n) \to \lambda$, but $T(H_n) \not \to \Pois(\lambda)$ (see Appendix \ref{sec:hypergraphexample} for an example). This raises the question: {\it What does one need to assume, in addition to \eqref{eq:ETHn}, to obtain a Poisson limit for $T(H_n)$, for any sequence of $r$-uniform hypergraph $H_n$, with $r \geq 3$?}

A few recent papers have addressed the above question in the context of graphs. Bhattacharya et al. \cite{bhattacharya2020second} established a second moment phenomenon for the number of monochromatic subgraphs in a randomly vertex-colored graph (which can be reformulated as a weighted analogue of $T(H_n)$, see Section \ref{sec:vertexsubgraph}). Specifically, \cite[Theorem 1.1]{bhattacharya2020second} shows that the number of monochromatic subgraphs in a sequence of graphs with vertices colored uniformly at random has a limiting Poisson distribution whenever its mean and variance converge to the same limit. In related work, \cite{distributionstars} characterized all distributional limits for counting monochromatic stars. In forthcoming paper \cite{subgraphpoisson}, the second moment phenomenon has also been established for monochromatic subgraphs in a randomly edge-
colored graph (recall Example \ref{example:edgecoloring}).

In this paper we generalize the above results in two directions: (1) from graphs to hypergraphs and (2) from marginal convergence to joint convergence. To begin with, we show that the Poisson approximation of $T(H_n)$ exhibits a second moment phenomenon, for any sequence of $r$-uniform hypergraphs $H_n$ with $r \geq 3$. More precisely, $T(H_n) \dto \Pois(\lambda)$, whenever $\bE T(H_n) \rightarrow \lambda$ and $\Var T(H_n) \rightarrow \lambda$ (see Theorem \ref{thm:THn}). Next, we show that the second moment phenomenon continues to hold for the joint distribution of the number of monochromatic hyperedges in 2-layer mutliplex hypergraphs (two hypergraphs sharing the same set of the vertices). More precisely, the joint distribution of $(T(H_n^{(1)}), T(H_n^{(2)}))$, where $H_n^{(1)}$ and $H_n^{(2)}$ are two hypergraphs on the same set of vertices, converges to a joint Poisson distribution, whenever their mean vector and covariance matrix converge. The marginal Poisson distributions are dependent when $H_n^{(1)}$ and $H_n^{(2)}$ are both $r$-uniform (see Theorem \ref{thm:hypergraphr}), whereas they are independent when $H_n^{(1)}$ and $H_n^{(2)}$ have different orders of uniformity (see Theorem \ref{thm:r1r2}). Extensions to more than 2 layers are also discussed (see Section \ref{sec:layersjoint}). Using these general results, we can derive  joint Poisson approximation results for the problems in Examples \ref{example:monochromaticedgesgraph}, \ref{example:edgecoloring}, and \ref{example:ap}. We discuss these in Section \ref{sec:applications} where we also obtain the joint distribution of the number of monochromatic hyperedges in a correlated Erd\H os-R\'enyi random hypergraph model and consider extensions to weighted hypergraphs, which, in particular, recovers the main result from \cite{bhattacharya2020second} about monochromatic subgraphs in a randomly vertex-colored graph.

\subsection{Second Moment Phenomenon for Hypergraphs} 
\label{sec:poissonhypergraph}

Our first result shows that the Poisson approximation for the number of monochromatic hyperedges is governed by a second moment phenomenon, for general $r$-uniform hypergraphs. 

\begin{thm}
    \label{thm:THn}
    Fix an integer $r \geq 2$ and consider a sequence $H_n$ of $r$-uniform hypergraphs. Suppose the vertices of $H_n$ are colored uniformly at random with $c_n$ colors as in \eqref{eq:cn}, such that the following hold: 
    $$\lim_{n \to \infty}\bE T(H_n) = \lambda \text{ and } \lim_{n \to \infty}\Var T(H_n) = \lambda, $$ 
    for some constant $\lambda \geq 0$. Then, as $n \rightarrow \infty$, $$T(H_n) \dto \Pois(\lambda).$$
\end{thm}

The proof of Theorem \ref{thm:THn} is essentially an adaption of the `truncated moment-comparison technique' from \cite{bhattacharya2020second} to the hypergraph setting.\footnote{Recall that \cite{bhattacharya2020second}  established the second moment phenomenon for monochromatic subgraphs in randomly vertex-colored graphs, which can be reformulated as a weighted  version of \eqref{eq:THn} (see Section \ref{sec:vertexsubgraph} for details).} In particular, the proof of Theorem \ref{thm:THn} has two main steps:

\begin{itemize}
\item {\it Truncation}: We decompose $T(H_n)$ into two parts: the {\it main term} and the {\it remainder term}. Informally, the remainder term counts the number of monochromatic hyperedges that share at least two vertices with a `large' number of other hyperedges. The first step in the proof is to show that this remainder term converges to zero in $L_1$. 

\item {\it Moment Comparison}: To analyze the {\it main term} we consider a surrogate obtained by replacing the collection of random variables $\{\bm 1 \{X_{=\bm e} \} : \bm e \in E(H_n)\}$ with independent $\Ber(\frac{1}{c_n^{r-1}})$ random variables, and show that the main term and the surrogate are asymptotically close in moments. 
\end{itemize}

\begin{remark} \label{remark:steinsmethod}
Another common approach to showing that a sequence of random variables has a Poisson limit is through the celebrated Chen-Stein method for Poisson approximation \cite{poisson2,barbour1992poisson,birthdayexchangeability,chatterjee2005exchangeable,barbour2005multivariate,poisson2023multivariate}.  In fact, the well-known dependency graph method \cite[Theorem 15]{chatterjee2005exchangeable} bounds the convergence rate in terms of the second moment  (but not in terms of the mean and variance). Arratia et al. \cite{arratia1990poisson}  (see also Chatterjee et al. \cite{chatterjee2005exchangeable}) used this to obtain the rate of convergence for the number of $r$-matching birthdays (recall the setup in Example \ref{example:monochromaticedgesgraph}). However, the dependency graph method cannot be directly used to prove Theorem \ref{thm:THn} for general graphs/hypergraphs, as the condition imposed by the convergence of the mean and variance is generally weaker than what is required by a generic dependency graph construction (see Remark \ref{remark:poissonhypergraph} for more details).  
\end{remark}

\subsection{Second Moment Phenomenon for 2-Layer Multiplex Hypergraphs}
\label{sec:poissonhypergraphlayers}

Higher-order networks \cite{battiston2021complexsystems,battiston2020networks,hypergraph2022higher,network2016higher} and networks with multiple layers are prototypical models for understanding complex relational data. In particular, there has been an explosion of interest in multilayer graphs (which are referred to as {\it multiplexes}) in recent years (see \cite{bianconi2018multilayer,networksbiology} for book-length treatments of multiplex networks). Multiplex hypergraphs (hypergraphs with more than one layer) are emerging as a new paradigm for modeling real-world networks with many interdependent higher-order interactions (see \cite{multiplexhypergrpahs,krishnagopal2023topology,lotito2024multiplex} and the references therein). 

Our aim in this section is to derive the joint distribution of the number of monochromatic hyperedges in a random vertex coloring of a hypergraph with two layers, hereafter referred to as a $2$-{\it hypermultiplex}. Formally, a 2-hypermultiplex is denoted as $\bm H_n = (H_n^{(1)}, H_n^{(2)})$, where $H_n^{(1)}$ is an $r_1$-uniform hypergraph and $H_n^{(2)}$ is an $r_2$-uniform hypergraph, for positive integers $r_1, r_2 \geq 2$, with a common vertex set $V(H_n^{(1)}) = V(H_n^{(2)}) = V_n$. Suppose the vertices in $V_n$ are colored with $c_n$ colors as in \eqref{eq:cn}. Then the bivariate analogue of \eqref{eq:THn} is defined as: 
\begin{align}
\label{eq:THn12}
\bm T( \bm{H}_n) = 
\begin{pmatrix} 
T(H_n^{(1)}) \\ 
T(H_n^{(2)}) 
\end{pmatrix}
 = 
 \begin{pmatrix} 
\sum_{\bm e \in E(H_n^{(1)})} \bm 1\{ X_{= \bm e}\} \\ 
\sum_{\bm e \in E(H_n^{(2)})} \bm 1\{ X_{= \bm e}\} 
\end{pmatrix} . 
\end{align} 
As in \eqref{eq:ETHn}, we are interested in the regime where $c_n \rightarrow \infty$ such that $\bE \bm T(H_n) = O(1)$.

The nature of limiting distribution of $\bm T(\bm H_n)$ depends on whether $r_1=r_2$ or $r_1 \ne r_2$. We first consider the case where $r_1=r_2$. In this case, $\bm T(\bm H_n)$ converges to a jointly dependent Poisson distribution, whenever the mean vector and the covariance matrix of $\bm T(\bm H_n)$ converges to appropriate limits. We formalize this in the following theorem.

\begin{thm}
  \label{thm:hypergraphr}
  Fix an integer $r \geq 2$ and consider a sequence $\bm H_n = (H_n^{(1)}, H_n^{(2)})$ of 2-hypermultiplexes, where both $H_n^{(1)}$ and $H_n^{(2)}$ are $r$-uniform. Suppose the vertices of $\bm H_n$ are colored uniformly at random with $c_n$ colors, such that the following hold: 
\begin{align}\label{eq:meanvariancer} 
\lim_{n \rightarrow \infty} \bE [ \bm T(\bm H_n) ] = \begin{pmatrix}
  \lambda_1 \\ 
  \lambda_2
  \end{pmatrix} \quad \text{ and } \quad \lim_{n \rightarrow \infty} \mathrm{Var}[ \bm T(\bm H_n) ] = \begin{pmatrix}
  \lambda_1 & \lambda_{1, 2} \\ 
  \lambda_{1, 2} & \lambda_2
  \end{pmatrix} , 
    \end{align}
   for constants $\lambda_1, \lambda_2, \lambda_{1,2} \geq 0$ with $\lambda_{1,2} \le \lambda_1$ and $\lambda_{1,2} \le \lambda_2$. Then
  $$ \bm T(\bm H_n) \dto \left(\begin{array}{l}
          Z_1 + Z_{1,2} \\
          Z_2 + Z_{1,2}
          \end{array}\right) , $$
  where $Z_1 \sim \Pois(\lambda_1 - \lambda_{1,2})$, $Z_2 \sim \Pois(\lambda_2 - \lambda_{1,2})$, and $Z_{1,2} \sim \Pois(\lambda_{1,2})$ are independent. 
\end{thm}

The proof of Theorem \ref{thm:hypergraphr} is given in Section \ref{sec:hypergraphrpf}, and it involves decomposing $T(H_n^{(1)})$ and $T(H_n^{(2)})$ in the following way: 
$$T(H_n^{(1)})= T(H_n^{(1)}\setminus H_n^{(2)}) + T(H_n^{(1)}\cap H_n^{(2)}), \quad T(H_n^{(2)})= T(H_n^{(2)}\setminus H_n^{(1)}) + T(H_n^{(1)}\cap H_n^{(2)})$$
where $H_n^{(1)}\setminus H_n^{(2)}$, $H_n^{(2)}\setminus H_n^{(1)}$, and $H_n^{(1)}\cap H_n^{(2)}$ denote the hypergraphs induced by the hyperedges that are only in $H_n^{(1)}$, the hyperedges that are only in $H_n^{(2)}$, and the hyperedges that are in both $H_n^{(1)}$ and $H_n^{(2)}$, respectively (see Section \ref{sec:hypergraphrpf} for the formal definitions). Then, using a truncation argument as in the proof of Theorem \ref{thm:THn} and a joint moment comparison (with an appropriate surrogate variable), we show that the joint distribution of 
$$(T(H_n^{(1)}\setminus H_n^{(2)}), T(H_n^{(2)}\setminus H_n^{(1)}), T(H_n^{(1)}\cap H_n^{(2)}))$$
converges to independent Poisson distributions with mean $\lambda_1 - \lambda_{1,2}, \lambda_2 - \lambda_{1,2}, \lambda_{1,2}$, respectively, under the assumptions of Theorem \ref{thm:hypergraphr}.

Next, we consider the case where the 2-layers $H_n^{(1)}$ and $H_n^{(2)}$ have different orders of uniformity ($r_1 < r_2$). In this case, since the 2 layers have different orders of uniformity, the ``intersection" hypergraph $H_n^{(1)}\cap H_n^{(2)}$ is empty, and the number of monochromatic hyperedges from layer 1 that are completely contained in a monochromatic hyperedge from layer 2 is asymptotically negligible. Hence, in this case  the asymptotic distribution of $\bm T(\bm H_n)$ has independent components. We formalize this result in the following theorem. The proof is given in  Section \ref{sec:r1r2pf}.  

\begin{thm}
  \label{thm:r1r2}
  Fix integers $2 \le r_1 < r_2$ and consider a sequence $\bm H_n = (H_n^{(1)}, H_n^{(2)})$ of 2-hypermultiplexes, where 
  $H_n^{(1)}$ is $r_1$-uniform and $H_n^{(2)}$ $r_2$-uniform. Suppose the vertices of $\bm H_n$ are colored uniformly at random with $c_n$ colors, such that the following hold:  For each $i \in \{1, 2\}$, 
\begin{align}\label{eq:meanvariancer1r2}
\lim_{n \to \infty}\bE[T(H_n^{(i)})] = \lim_{n \to \infty}\Var[T(H_n^{(i)})] = \lambda_i . 
\end{align} 
      for some constant $\lambda_i \geq 0$. Then
  $$ \bm T(\bm H_n) \dto \left(\begin{array}{l}
          Z_1\\
          Z_2
          \end{array}\right) , $$
    where $Z_1 \sim \Pois(\lambda_1), Z_2 \sim \Pois(\lambda_2)$ are independent.
\end{thm}

Theorems \ref{thm:hypergraphr} and \ref{thm:r1r2} combined show  that for multiplex hypergraphs with 2 layers, the joint distribution of the number of monochromatic hyperedges satisfies a second moment phenomenon. Specifically, $\bm T(\bm H_n)$ has a joint (possibly dependent) Poisson distribution whenever the mean vector and the covariance matrix of $\bm T(\bm H_n)$ converges. This extends the marginal second moment phenomenon in Theorem \ref{thm:THn} to the bivariate case.

\subsection{More Than 2 Layers}
\label{sec:layersjoint}

In this section we discuss generalizations of the previous results to multiplex hypergraphs with more than 2 layers. Formally, a $d$-{\it hypermultiplex}, for an integer $d \geq 2$, will be denoted by $\bm H_n = (H_n^{(1)}, H_n^{(2)}, \ldots, H_n^{(d)})$, where $H_n^{(i)}$ is a $r_i$-uniform hypergraph, for a positive integer $r_i \geq 2$, with a common vertex set $V(H_n^{(i)}) = V_n$, for $1 \leq i \leq d$. The following result provides a general condition under which the joint distribution of the numbers of monochromatic hyperedges in a $d$-hypermultiplex has a multivariate Poisson limit with independent components. 

\begin{prop} 
    \label{ppn:dmultiplegraphs} 
    Fix $d \geq 1$ and integers $2 \leq r_1 \leq r_2 \cdots \leq r_d$. Let $\bm H_n = (H_n^{(1)}, \ldots, H_n^{(d)})$ be a sequence of $d$-layered multiplex hypergraphs, where $H_n^{(i)}$ is $r_i$-uniform, for $1 \leq i \leq d$. Suppose the vertices of $\bm{H}_n$ are colored uniformly at random with $c_n$ colors as in \eqref{eq:cn} such that the following hold:  
    
\begin{itemize}
\item For $1 \leq i \leq d$, 
    \begin{align}\label{eq:TnHd}
    \lim_{n \to \infty}\bE T_n(H_n^{(i)}) = \lim_{n \to \infty}\Var T_n(H_n^{(i)}) = \lambda_i , 
    \end{align}
    for some constant $\lambda_i \geq 0$.
    
\item For $1 \leq i \ne j \leq d$, 
\begin{align}\label{eq:EHncommon}
E(H_n^{(i)}) \cap E(H_n^{(j)}) = \emptyset . 
\end{align} 
    \end{itemize}
    Then, as $n \rightarrow \infty$, 
    $$ \bm T(\bm H_n) := \left(\begin{array}{c}
        T(H_n^{(1)}) \\ 
        \vdots\\
        T(H_n^{(d)})
        \end{array}\right) \dto \left(\begin{array}{c}
            Z_1 \\ 
            \vdots \\
            Z_d
            \end{array}\right) := \bm Z, $$
            where $Z_1\sim \Pois(\lambda_1), \ldots, Z_d \sim \Pois(\lambda_d)$ are independent. 
\end{prop} 

The proof of Proposition \ref{ppn:dmultiplegraphs} is given in Section \ref{sec:layersjointpf}. The main observation is that under the condition \eqref{eq:EHncommon}, the contributions of the number of monochromatic hyperedges that are common between two layers are negligible. The proof technique of Theorem \ref{thm:r1r2} can then be applied to show that the limiting joint distribution has independent components.

Note that if $H_n^{(1)}$ is $r_1$-uniform and $H_n^{(2)}$ is $r_2$-uniform, with $r_1 \ne r_2$, then $E(H_n^{(1)}) \cap E(H_n^{(2)}) = \emptyset$. Hence, in this case, \eqref{eq:EHncommon} holds automatically, and we have the following corollary.    

\begin{cor}
    \label{cor:dmultiplegraphsdistinct} 
    Fix $d \geq 1$ and distinct integers $r_1, \ldots, r_d \geq 2$. Let $\bm H_n = (H_n^{(1)}, \ldots, H_n^{(d)})$ be a sequence of $d$-hypermultiplexes, where $H_n^{(i)}$ is $r_i$-uniform, for $1 \leq i \leq d$. Suppose the vertices of $\bm{H}_n$ are colored uniformly at random with $c_n$ colors as in \eqref{eq:cn} such that  \eqref{eq:TnHd} holds, for $1 \leq i \leq d$. Then, as $n \rightarrow \infty$, 
    $$\bm T(\bm H_n) \dto \bm Z , $$ 
where $ \bm T(\bm H_n)$ and $\bm Z$ are as defined in Proposition \ref{ppn:dmultiplegraphs}. \hfill $\Box$
\end{cor}

The remaining case is when \eqref{eq:EHncommon} does not hold. For instance, consider a sequence of $d$-hypermultiplexes $\bm{H}_n$ with $d \geq 3$, where all layers are $r$-uniform. %What can one say about the joint distribution of $\bm T(\bm H_n)$ when the number of edges common across pairs of layers is $\Theta(c_n^{r-1})$?\footnote{For two non-negative sequences $a_n$ and $b_n$, we say $a_n = \Theta(b_n)$, if, for all sufficiently large $n$, $C_1 b_n \leq a_n \leq C_2 b_n$, for constants $C_1, C_2 > 0$. }  
In this case, it is easy to construct examples where the second-moment phenomenon fails, that is, only assuming the convergence of the mean vector and the covariance matrix is not enough to ensure the joint distributional convergence of the number monochromatic hyperedges in $d$-hypermultiplexes, when $d \geq 3$ and all the layers are $r$-uniform (see Appendix \ref{sec:3layersH}). In such cases, additional conditions in terms of higher moments (beyond the means and covariances) are required for distributional convergence. We leave this for future research. 
%In fact, using the techniques in this paper it is possible to envisage a result about the joint convergence of monochromatic hyperedges in hypermultiplexes with 3 or more layers, assuming the convergence the higher-order central moments.  

%However, their distributions can diverge, as their intersections require additional characterizations beyond the convergence of means and covariances.

\subsection{Applications} 
\label{sec:applications}

In this section, we revisit the examples mentioned in the Introduction and discuss other applications of the previous results.  
The section is organized as follows: We begin by discussing the case of graphs ($r=2$) in Section \ref{sec:graphjoint}. The joint distribution of the number of monochromatic subgraphs in randomly edge-colored graphs (recall Example \ref{example:edgecoloring}) is derived in Section \ref{sec:jointedge}. The Poisson approximation for monochromatic APs in randomly colored integers (recall Example \ref{example:ap}) is discussed in Section \ref{sec:jointap}. We apply Theorem \ref{thm:hypergraphr} to derive the asymptotic joint distribution of monochromatic hyperedges in a correlated Erd\H os-R\'enyi random hypergraph model in Section \ref{sec:randomhypergraphs}. Finally, in Section \ref{sec:vertexsubgraph} we discuss extensions to weighted graphs and monochromatic subgraphs in random vertex colorings (recall Example \ref{example:monochromaticedgesgraph}). 

\subsubsection{Joint Distribution for 2-Layered Multiplex Graphs} 
\label{sec:graphjoint}

We begin with the case $r=2$, that is, when $H_n$ is a graph. In this case, it is known that $T(H_n)$ satisfies a first moment phenomenon (recall the discussion following \eqref{eq:ETHn}). One can recover this result from Theorem \ref{thm:THn} by observing that when $r=2$, 
$$\Var T(H_n) =\frac{|E(H_n)|}{c_n}\Big(1-\frac{1}{c_n}\Big).$$ 
Therefore, convergence of the first moment $\bE T(H_n) = \frac{|E(H_n)|}{c_n} \rightarrow \lambda$, automatically ensures that $\Var T(H_n) \rightarrow \lambda$. 

The next corollary provides an extension of this result to graphs with two layers (which we refer to as a 2-multiplex). The proof is an immediate consequence of Theorem \ref{thm:r1r2}, the discussion above, and the assertion in \eqref{eq:EH12}. 

\begin{cor}\label{cor:graph}
Consider a sequence $\bm H_n = (H_n^{(1)}, H_n^{(2)})$ of 2-multiplexes. Suppose the vertices of $\bm H_n$ are colored uniformly at random with $c_n$ colors as in \eqref{eq:cn}, such that the following hold: 
\begin{itemize}
\item $\lim_{n \rightarrow \infty}  \frac{|E(H_n^{(1)})|}{c_n} = \lambda_1$ and $\lim_{n \rightarrow \infty}  \frac{|E(H_n^{(2)})|}{c_n} = \lambda_2$,

\item $\lim_{n \rightarrow \infty} \frac{|E(H_n^{(1)}) \cap E(H_n^{(2)})|}{c_n} = \lambda_{1, 2}$, 
\end{itemize} 
for constants $\lambda_1, \lambda_2, \lambda_{1, 2} \geq 0$, with $\lambda_{1,2} \le \lambda_1$ and $\lambda_{1,2} \le \lambda_2$. Then
  $$ \bm T(\bm H_n) \dto \left(\begin{array}{l}
          Z_1 + Z_{1,2} \\
          Z_2 + Z_{1,2}
          \end{array}\right) , $$
  where $Z_1 \sim \Pois(\lambda_1 - \lambda_{1,2})$, $Z_2 \sim \Pois(\lambda_2 - \lambda_{1,2})$, and $Z_{1,2} \sim \Pois(\lambda_{1,2})$ are independent. 
\end{cor}

\subsubsection{Monochromatic Subgraphs in Random Edge Colorings}
\label{sec:jointedge}

Given a fixed graph $F=(V(F), E(F))$ and a sequence of graphs $G_n = V(G_n, E(G_n))$, recall from Example \ref{example:edgecoloring} that $R(F, G_n)$ denotes the number of monochromatic copies of $F$ when the edges of $G_n$ are colored uniformly at random with $c= c_n$ colors. The second moment phenomenon for $R(F, G_n)$, for any fixed graph $F$ and any sequence of graphs $G_n$, has been established in \cite{subgraphpoisson}. The following result extends this result to the multivariate setting: 
\begin{cor}\label{cor:jointedge} Fix $d\geq 1$ and a finite collection of non-isomorphic graphs $F_1, F_2, \ldots, F_d$. Suppose $G_n = ( V(G_n), E(G_n))$ is a sequence of graphs with edges colored uniformly at random with $c_n$ colors such that the following holds: For all $1 \leq i \leq d$, there exists a constant $\lambda_i \geq  0$ such that 
    $$\lim_{n \to \infty}\bE R(F_i, G_n) = \lim_{n \to \infty}\Var R(F_i, G_n) = \lambda_i. $$ 
      Then, as $n \rightarrow \infty$, 
    $$\left(\begin{array}{c}
        R(F_1, G_n) \\
       	\vdots \\ 
	R(F_d, G_n) 
        \end{array}\right) \dto \left(\begin{array}{c}
            Z_1 \\ 
            \vdots \\ 
            Z_d
            \end{array}\right) , $$
      where $Z_1 \sim \Pois(\lambda_1), \ldots, Z_d \sim \Pois(\lambda_d)$ are independent. \hfill $\Box$
\end{cor}

The proof of Corollary \ref{cor:jointedge} is given Section \ref{sec:jointedgepf}. The proof uses the representation $R(F_i, G_n) = T_n(H_{G_n}(F_i))$ from Example \ref{example:edgecoloring} and Proposition \ref{ppn:dmultiplegraphs}. 

\begin{remark} In forthcoming work \cite{subgraphpoisson} we show that $R(F, G_n)$ exhibits an intriguing first moment phenomenon for certain graphs $F$, where the convergence of the expectation alone is sufficient to obtain a Poisson limit. For Corollary \ref{cor:jointedge} to hold for a finite collection of such graphs, the convergence of the mean vector alone suffices. 
\end{remark}

\begin{remark}
The number of monochromatic subgraphs can also have a Poisson behavior in a different asymptotic regime where the number of colors is fixed, but the size of the graph $F$ grows with $n$ in a way that the random variable has bounded mean. Specifically, when $c=2$, Godbole et al.~\cite{godbole1995asymptotic} showed that $R(K_r, K_n)$ has a Poisson limit when $r=r_n \rightarrow \infty$ such that $\bE(R(K_r, K_n)) = \Theta(1)$. 
\end{remark}

\subsubsection{Monochromatic Arithmetic Progressions} 
\label{sec:jointap}

Suppose the elements of a set $A_n \subseteq [n]$, with $|A_n| \rightarrow \infty$, are colored independently and uniformly at random with $c_n$ colors. Recall from Example \ref{example:ap} that $W_r(A_n)$ is the number of monochromatic $r$-APs in $A_n$. Note that 
\begin{align}\label{eq:Anr}
\b E W_r(A_n)  = \frac{N_r(A_n)}{c_n^{r-1}} , 
\end{align}
where $N_r(A_n)$ is the number of $r$-APs in $A_n$. 
In this section we derive the joint distribution of $(W_3(A_n), W_4(A_n), \ldots , W_d(A_n))$, for any $d \geq 3$. To this end, recall that 
$W_r(A_n) = T(H_{A_n}(r))$, where $H_{A_n}(r)$ is the $r$-uniform hypergraph defined in Example \ref{example:ap}. Since for $r_1 \ne r_2$, the corresponding hypergraphs $H_{A_n}(r_1)$ and $H_{A_n}(r_2)$ have different orders of uniformity, a direct application of Corollary \ref{cor:dmultiplegraphsdistinct} yields the following result:

\begin{cor}\label{cor:jointap} 
Fix $d\geq 3$. Suppose  the elements of a set $A_n \subseteq [n]$, with $|A_n| \rightarrow \infty$, are colored independently and uniformly at random with $c_n$ colors such that the following holds: For all $3 \leq r \leq d$, there exists a constant $\lambda_r \geq 0$ with  
    \begin{align}\label{eq:An}
    \lim_{n \to \infty}\bE W_r(A_n) = \lim_{n \to \infty}\Var W_r(A_n) = \lambda_r. 
    \end{align} 
      Then, as $n \rightarrow \infty$, 
    $$\left(\begin{array}{c}
        W_3(A_n) \\
       	\vdots \\ 
	W_d(A_n) 
        \end{array}\right) \dto \left(\begin{array}{c}
            Z_3 \\ 
            \vdots \\ 
            Z_d
            \end{array}\right), $$
      where $Z_3 \sim \Pois(\lambda_1), \ldots, Z_d \sim \Pois(\lambda_d)$ are independent.  \hfill $\Box$ 
\end{cor} 

When $A_n = [n]$, observe from \eqref{eq:Anr} that to obtain a non-trivial limit for $W_r([n])$ one has to choose $c_n = \Theta(n^{\frac{2}{r-1}})$, since the number of $r$-APs in $[n]$ is $\Theta(n^{2})$, for all $r \geq 3$. 
With this scaling, it turns out that the variance condition in \eqref{eq:An} holds automatically. This leads to the following result (see Section \ref{sec:apnpf} for the proof).

\begin{cor}\label{cor:apn}
Fix an integer $r \geq 3$. Suppose the elements in $[n]$ are colored independently and uniformly at random with $c_n=\lambda n^{\frac{2}{r-1}}$ colors, where $\lambda > 0$ is a positive constant. Then $W_r([n]) \dto \Pois(\frac{\lambda}{r-1})$. 
\end{cor}

\subsubsection{Correlated Erd\H{o}s-R\'enyi Hypergraphs}
\label{sec:randomhypergraphs}

The prototypical example of a random $r$-uniform hypergraph is 
the Erd\H{o}s-R\'enyi model $H_r(n, p)$, which is a hypergraph with vertex set $[n]$ where every $r$-element subset of $[n]$ is present as a hyperedge independently with probability $p = p_n \in (0, 1)$. A natural extension of this to the multiplex setting is the {\it correlated Erd\H{o}s-R\'enyi random hypergraph model} $\bm H_r(n, p, \rho)$, which is a 2-hypermultiplex where the hyperedges are dependent across the different layers. Specifically, $\bm H_r(n, p, \rho) = (H_n^{(1)}, H_n^{(2)})$ is a 2-hypermultiplex with common vertex set $[n]$, where independently for every $r$-element subset $\bm e \in {[n] \choose r}$, we have 
$$\bP(\bm e \in E(H_n^{(1)})) = \bP(\bm e \in E(H_n^{(2)})) = p  \text{ and }  \bP(\bm e \in E(H_n^{(1)}), \bm e \in E(H_n^{(2)})) =  \rho + p^2 := p_{1, 2},$$ 
for $p= p_n \in (0, 1)$ and $\rho = \rho_n \in [0, p(1-p))$. 

Theorems \ref{thm:THn} and \ref{thm:hypergraphr} can be easily extended to random hypergraphs, when the limits in \eqref{eq:meanvariancer} and \eqref{eq:meanvariancer1r2} hold in probability, under the assumption that the hypergraph and its coloring are jointly independent (see Lemma \ref{lm:randomhypergraph}). Using this we can derive the distribution $\bm T(\bm H_n)$, where the vertices of $\bm H_n = (H_n^{(1)}, H_n^{(2)}) \sim \bm H_r(n, p, \rho)$ are colored with $c_n$ colors such that 
\begin{align}\label{eq:p12}
\bE T(H_n^{(1)}) = \bE T(H_n^{(2)}) = \frac{\binom{n}{r}p}{c_n^{r-1}} \to \lambda \quad \text{ and } \quad \frac{\bE[|E(H_n^{(1)} \cap H_n^{(2)})|]}{c_n^{r-1}} = \frac{{n \choose r} p_{1, 2} }{c_n^{r-1}} \rightarrow \lambda_{1, 2}, 
\end{align}
for constants $0 \leq \lambda_{1, 2} < \lambda$. 
Under this assumption we have the following result. The proof is given in Section \ref{sec:randomhypergraphpf}. 

\begin{cor}\label{cor:randomhypergraph}
   Fix $r \geq 2$ and let $\bm H_n = (H_n^{(1)}, H_n^{(2)}) \sim \bm H_r(n, p, \rho)$ with parameters $p= p_n \in (0, 1)$ and $\rho = \rho_n \in [0, p(1-p))$. Given $\bm H_n$, suppose the vertices of $\bm H_n$ are colored uniformly at random with $c_n$ colors as in \eqref{eq:cn} such that \eqref{eq:p12} holds. Then for $p \gg n^{-r}$, we have
    $$ \bm T(\bm H_n) \dto \left(\begin{array}{l}
            Z_1 + Z_{1,2}\\
            Z_2 + Z_{1,2}
            \end{array}\right) , $$
    where $Z_1 \sim \Pois(\lambda - \lambda_{1, 2})$, $Z_2 \sim \Pois(\lambda - \lambda_{1, 2})$, $Z_{1,2} \sim \Pois(\lambda_{1, 2})$ are independent. 
\end{cor}

\subsubsection{Weighted Graphs and Monochromatic Subgraphs in Random Vertex Colorings} 
\label{sec:vertexsubgraph}  

The results in Sections \ref{sec:poissonhypergraph} and \ref{sec:poissonhypergraphlayers} can be easily extended to  hypergraphs with integer valued weights. To this end, suppose that $H_n$ is a sequence of $r$-uniform hypergraphs where each edge $\bm e \in E(H)$ is assigned a weight $w_{\bm e}$, where $w_{\bm e}$ is some integer in $[1, K]$, for some positive integer $K \geq 1$. Define $\mathsf{W}(H_n)$ as the weighted sum of the number of monochromatic hyperedges in $H_n$ in a uniformly random $c_n$-coloring as in \eqref{eq:cn}. That is, 
\begin{align}\label{eq:WHn}
\mathsf{W}(H_n) = \sum_{\bm e \in E(H_n)} w_{\bm e} \bm 1\{X_{=\bm e}\}, 
\end{align} 
where $w_{\bm e}$ is an integer in $[1, K]$. (Note that when $K=1$, or, $w(\bm e) = 1$ for all $\bm e \in E(H_n)$, we have $\mathsf{W}(H_n) = T(H_n)$.) For $1 \leq i \leq K$, let $H_n^{(i)}= (V(H_n^{(i)}, E(H_n^{(i)}))$, where 
\begin{align}\label{eq:weightlayers}
V(H_n^{(i)}) = V(H_n) \text{ and } E(H_n^{(i)}) = \{\bm e \in E(H_n) : w_{\bm e} = i\} , 
\end{align}
and consider the $K$-hypermultiplex $\bm H_n= (H_n^{(1)}, H_n^{(2)}, \ldots, H_n^{(K)})$. Then \eqref{eq:WHn} can be written as: 
$$\mathsf{W}(H_n) = \sum_{i = 1}^K i \cdot T(H_n^{(i)}).$$
Note the hyperedge sets of $H_n^{(1)}, H_n^{(2)}, \ldots, H_n^{(K)}$ are mutually disjoint; hence, from Proposition  \ref{ppn:dmultiplegraphs} we have the following result:

\begin{cor}
    \label{cor:WHn}
    Fix integers $r \ge 2$ and $K \geq 1$. Let $\{H_n\}_{n\ge 1}$ be a sequence of weighted $r$-uniform hypergraphs as described above with vertices colored uniformly with $c_n$ colors as in \eqref{eq:cn}. Suppose, for each $ 1 \leq i \leq K$,
    $$\lim_{n \to \infty}\bE T(H_n^{(i)}) = \lambda_i \text{ and } \lim_{n \to \infty}\Var T(H_n^{(i)}) =   \lambda_i , $$ 
    where $H_n^{(i)}= (V(H_n^{(i)}, E(H_n^{(i)}))$ is as defined in \eqref{eq:weightlayers} and $\lambda_i \geq 0$ is a constant. Then, as $n \rightarrow \infty$,  
    $$\mathsf{W}(H_n) \dto \sum_{i=1}^K i\cdot Z_i,$$ where $Z_i \sim \Pois(\lambda_i)$ are independent, for $1 \leq i \leq K$. \hfill $\Box$ 
\end{cor}

The above corollary implies a second moment phenomenon for  $\mathsf{W}(H_n)$ (see Section \ref{sec:WHnlambdapf} for the proof).

\begin{cor}
    \label{cor:WHnlambda} 
    Let $\{H_n\}_{n\ge 1}$ be a sequence of weighted $r$-uniform hypergraphs as described above with vertices colored uniformly with $c_n$ colors as in \eqref{eq:cn}. Suppose 
    \begin{align}\label{eq:lambdaRHn}
    \lim_{n \to \infty}\bE \mathsf{W}(H_n) = \lambda \text{ and } \lim_{n \to \infty}\Var \mathsf{W}(H_n) = \lambda, 
    \end{align}
    for some $\lambda \ge 0$. Then, as $n \rightarrow \infty$, $\mathsf{W}(H_n) \dto \Pois(\lambda)$. 
\end{cor}

The framework of weighted hypergraphs can be applied to derive the distribution of the number of monochromatic subgraphs in a uniformly random vertex coloring of a sequence of graphs. To this end, suppose $G_n = (V(G_n), E(G_n))$ is a sequence of graphs such that the vertices of $G_n$ are colored uniformly at random with $c_n$ colors as in \eqref{eq:cn}. For a fixed graph $F = (V(F), E(F))$, let $Q(F, G_n)$ denote the number of monochromatic copies of $F$ in $G_n$, where we say $F$ is {\it monochromatic} if all the vertices of $F$ have the same color.\footnote{Recall that in Section \ref{sec:jointedge}, we considered the problem of counting the number of monochromatic copies of a graph $F$ in a randomly {\it edge} colored graph $G_n$. In that context, $F$ is said to be monochromatic if all the {\it edges} of $F$ have the same color. Here, we are considering the analogous problem for randomly vertex-colored graphs. } 
Formally,  
$$Q(F, G_n):=\frac{1}{|Aut(F)|}\sum_{\bm s \in V(G_n)_{|V(F)|}} \prod_{(a,b) \in E(F)}a_{s_a s_b}(G_n) \bm 1\{X_{=\bm s}\},$$
where:  
\begin{itemize}
\item[--] $V(G_n)_{|V(F)|}$ is the set of all $|V(F)|$-tuples $\bm s = (s_1, s_2, \ldots, s_{|V (F)|} ) \in V(G_n)^{|V(F)|}$ with distinct
entries,\footnote{For any finite set $S$ and positive integer $k \geq 1$, $S^k = S \times S \cdots \times S$  denotes the $k$-fold Cartesian product of $S$.}

\item[--] for any ${\bm s}=(s_1,\cdots, s_{|V(F)|}) \in  V(G_n)_{|V(F)|}$, 
\begin{align*}%\label{eq:sind}
\bm 1\{X_{=\bm s}\}:= \bm 1\{X_{s_1}=\cdots=X_{s_{|V(F)|}}\},  
\end{align*}

\item[--] $Aut(F)$ is the {\it automorphism group} of $F$, that is, the set of permutations $\sigma$ of the vertex set $V(F)$ such that $(x, y) \in E(F)$ if and only if $(\sigma(x), \sigma(y)) \in E(F)$. 
\end{itemize}
Note that when $F=K_2$, $Q(K_2, G_n)$ is the number of monochromatic edges in a random vertex coloring $G_n$, which arises in the study of the birthday problem (recall Example \ref{example:monochromaticedgesgraph}). To express $Q(F, G_n)$ in terms of \eqref{eq:WHn}, define the weighted $|V(F)|$-uniform hypergraph $\bar{H}_{G_n}(F) = (V(\bar{H}_{G_n}(F)), E(\bar{H}_{G_n}(F)) )$ as follows: 

\begin{itemize}
    \item[--] $V(\bar{H}_{G_n}(F)) = V(G_n)$.
    \item[--] For each $|V(F)|$-tuple $\bm s \in { V(G_n) \choose |V(F)| }$, denote by $G_n[{\bm s}]$ the induced subgraph of $G_n$ on the vertices in $\bm s$ and $N(F, G_n[{\bm s}])$ be the number of copies of $F$ in the induced graph $G_n[\bm s]$. Then 
    $$E(\bar{H}_{G_n}(F)) = \left\{ \bm s \in { V(G_n) \choose |V(F)| }:  N(F, G_n[{\bm s}]) \geq 1 \right\} $$  
    and $w(\bm e) = N(F, G_n[{\bm s}])$, for $\bm e \in E(\bar{H}_{G_n}(F))$. 
\end{itemize}
With this construction, it is easy to check that $Q(F, G_n) = \mathsf{W}(\bar{H}_{G_n}(F))$. Hence, invoking Corollary \ref{cor:WHn} we recover the second moment phenomenon for $Q(F, G_n)$ from \cite{bhattacharya2020second}.\footnote{Note that when $F=K_r$ is the $r$-clique, then $\bar{H}_{G_n}(K_r)$ is the same as the hypergraph $\cC_r(G_n)$ defined in Example \ref{example:monochromaticedgesgraph}. In this case, the hypergraph is unweighted, since $N(K_r, G_n[{\bm s}])$ is either 0 or 1, for any $\bm s \in { V(G_n) \choose r}$. }

\begin{cor}[{\cite[Theorem 1.1]{bhattacharya2020second}}]
    \label{cor: subgraph counting} 
    Fix a graph $F= (V(F), E(F))$. Let $G_n$ be a sequence of graphs with vertices colored uniformly at random with $c_n$ colors as in \eqref{eq:cn} such that
    $$\lim_{n \rightarrow \infty}\bE Q(F, G_n)=\lambda \quad \text{and}\quad \lim_{n \rightarrow \infty}\Var Q(F, G_n)=\lambda ,$$ 
    for some constant $\lambda \geq 0$. Then $Q(F, G_n) \dto \Pois(\lambda)$. \hfill $\Box$  
\end{cor}

\begin{remark} 
The limiting distribution of $Q(F, G_n)$ in the regime where $c_n = c$ is fixed has also been studied in a series of recent papers. In this case, the Gaussian limit of $Q(F, G_n)$ (appropriately centered and scaled) is often governed by a fourth moment phenomenon \cite{bhattacharya2017universal,bhattacharya2022normal,asymptoticdistributionrandomquadratic,subgraph2023fourth, subgraph2024characterizing}. Non-Gaussian limits can also arise, for instance, when $G_n$ is a converging sequence of dense graphs \cite{bhattacharya2019monochromatic}. 
A characterization of all distributional limits of $Q(K_2, G_n)$ when $c=2$ is provided in \cite{asymptoticdistributionrandomquadratic}. 
\end{remark}

\section{ Proof of Theorem \ref{thm:THn} }
\label{sec: marginal} 

We begin by introducing a definition that will be recurrent throughout our discussion: 
\begin{defn}
    \label{defn:K}
  Suppose $H_n$ is an $r$-uniform hypergraph. For an integer $t \le r - 1,$ we define $\cK(t, H_n)$ as the set of ordered pairs of hyperedges $(\bm e_1, \bm e_2)$ that share {\it exactly} $t$ vertices, that is, $\bm e_1, \bm e_2 \in E(H_n)$ such that $|\bm e_1 \cap \bm e_2| = t$. 
\end{defn}

Theorem \ref{thm:THn} is a consequence of the following result which gives sufficient conditions in terms of $\cK(t, H_n)$ under which $T(H_n)$ converges to a Poisson distribution. 

\begin{thm} 
    \label{thm:THnKpoisson}
  Fix an integer $r \geq 2$ and consider a sequence $H_n$ of $r$-uniform hypergraphs with vertices colored uniformly at random with $c_n$ colors as in \eqref{eq:cn}, such that the following hold:
  \begin{itemize}
      \item $\lim_{n\to \infty}\frac{|E(H_n)|}{c_n^{r-1}} = \lambda$
      \item For $t \in [2, r - 1]$, $|\cK(t, H_n)| = o(c_n^{2r-t-1})$.
  \end{itemize}
  Then $T(H_n) \dto \Pois(\lambda)$. 
\end{thm}

The proof of Theorem \ref{thm:THnKpoisson} is given in Section \ref{sec:poissonmarginalpf}. We now derive Theorem \ref{thm:THn} from Theorem \ref{thm:THnKpoisson}. To do this, suppose the conditions of Theorem \ref{thm:THn} hold. Then 
$$\bE T(H_n) = \frac{|E(H_n)|}{c_n^{r-1}} \to \lambda.$$ 
Also, $\Var T(H_n) = R_{1,n} + R_{2,n}, $ 
where 
\begin{align}\label{eq:varianceTHn}
R_{1,n} = \frac{1}{c_n^{r-1}}\left(1 - \frac{1}{c_n^{r-1}}\right)|E(H_n)| \text{ and } R_{2,n} = \sum_{t=2}^{r-1}\frac{1}{c_n^{2r-t-1}}\left(1 - \frac{1}{c_n^{t-1}}\right)|\cK(t, H_n)| . 
\end{align} 
Since $R_{1,n} \to \lambda$ and $\Var T(H_n) \to \lambda$ by the assumptions of Theorem \ref{thm:THn}, we have $R_{2,n} \to 0$. This implies each of the summands in $R_{2,n}$ should converge to zero, or $|\cK(t, H_n)| = o(c_n^{2r-t-1})$, for $t \in [2, r - 1]$. Hence, the conditions of Theorem \ref{thm:THnKpoisson} are satisfied and $T(H_n) \dto \Pois(\lambda)$. 

%
%\begin{remark}
%    For all the results discussed here, we can replace ``at least" in bold in Definition \ref{defn:K} by ``exactly" as we observe $\sum_{s = t}^{r-1} o(c_n^{2r-s-1}) = o(c_n^{2r-t-1})$.
%\end{remark}
%

\subsection{Proof of Theorem \ref{thm:THnKpoisson} }
\label{sec:poissonmarginalpf}

We first provide the proof assuming $r \ge 3$, as this allows us to avoid some degeneracies that arise when $r=2$. After presenting the proof for $r \ge 3$, we will return to the $r=2$ case in Section \ref{sec:THnKgraphpf}.

%We begin with a few notations and definitions. For any finite set $S$ and positive integer $k \geq 1$, let $S^k = S \times S \cdots \times S$  be the $k$-fold Cartesian product of $S$ and $S_k$ be the collection of all $k$-tuples $\bm s = (s_1, s_2, \ldots, s_k) \in S^k$ with distinct entries. Also, for any ordered $k$-tuple $\bm s \in S_k$, denote by $\bm e$ the (unordered) set formed by the entries of $\bm s$ (for example, if $\bm s = (4, 2, 5)$, then $\bm e = \{2, 4, 5\}$). 

\subsubsection{Proof of Theorem \ref{thm:THnKpoisson} when $r \geq 3$ }
\label{sec:THnKhypergraphpf}

For $t \in [2, r-1]$ and $\bm s = \{s_1, \ldots , s_t \} \in {V(H_n) \choose t}$, let $M_t(\bm s, H_n)$ denote the set of hyperedges in $H_n$ containing the indices in $\bm s$, that is, 
    \begin{align}\label{eq:MsHn}
    M_t(\bm s, H_n) = \{ \bm e \in E(H_n) : \bm s \subset \bm e\} , 
    \end{align} 
Now, fixing $\varepsilon > 0$, define 
$$\cA_{\varepsilon, t}(H_n) := \{\bm e \in E(H_n): |M_t(\bm s, H_n)| \le \varepsilon c_n^{r-t} \text{ for all } \bm s \subset \bm e \text{ such that } |\bm s| = t \} $$ 
and 
    \begin{align}\label{eq:AepsilonHn}
    \cA_{\varepsilon}(H_n) = \bigcap_{t=2}^{r-1} \cA_{\varepsilon, t}(H_n). 
    \end{align}
Intuitively, $\cA_{\varepsilon, t}(H_n)$ is the set of hyperedges in $H_n$ whose intersection with other hyperedges (in a subset of size $t \in [2, r-1]$) is ``small.'' Using the set $\cA_{\varepsilon}(H_n)$ we decompose $T(H_n)$ into two terms: the {\it main term}    
    \begin{align}\label{eq:THnepsilon}
    T_{\varepsilon}^+(H_n) = \sum_{\bm e \in  \cA_{\varepsilon}(H_n)}\bm 1\{X_{=\bm e}\}
    \end{align}
and the {\it remainder term} 
    \begin{align}\label{eq:Tremainder}
    T_{\varepsilon}^-(H_n) = T(H_n) - T_{\varepsilon}^+(H_n) = \sum_{\bm e \not \in \cA_{\varepsilon}(H_n)}\bm 1\{X_{=\bm e}\} . 
    \end{align}

The first step in the proof of Theorem \ref{thm:THnKpoisson} is to show that $T_{\varepsilon}^-(H_n)$ is negligible in $L_1$, for any fixed $\varepsilon > 0$. For two non-negative sequences $a_n$ and $b_n$, we use the notation $a_n \lesssim_{\square} b_n$ to denote that $a_n \le C(\square) \cdot b_n$ where $C(\square) > 0$ is a constant depending on the subscripted parameters (denoted by $\square$).  

\begin{lem}
    \label{lm:remainderH}
    For each fixed $\varepsilon > 0$, $T_{\varepsilon}^-(H_n) \overset{L_1}{\to} 0$, as $n \to \infty$.
\end{lem} 

\begin{proof}
    Notice from \eqref{eq:Tremainder} that 
    $$\bE T_{\varepsilon}^-(H_n) = \sum_{\bm e \not \in \cA_{\varepsilon}(H_n)} \bP(X_{=\bm e} ) = \frac{1}{c_n^{r-1}}|E(H_n)\setminus\cA_{\varepsilon}(H_n)|.$$
    By definition and the union bound,
    \begin{align*}
        |E(H_n)\setminus\cA_{\varepsilon}(H_n)| 
        &\le\sum_{\bm e \in E(H_n)}\sum_{t=2}^{r-1} 
        \bm 1\{ \text{there is } \bm s \subset \bm e \text{ with } |\bm s| =t \text{ such that } M_t(\bm s, H_n) > \varepsilon c_n^{r-t}\}\\
        &=\sum_{t=2}^{r-1}\sum_{\bm s \in {V(H_n) \choose t} } \bm 1\{ \bm e \in E(H_n), \bm s \subset \bm e, M_t(\bm s, H_n) > \varepsilon c_n^{r-t}\}\\
        &\le\sum_{t=2}^{r-1}\sum_{\bm s \in {V(H_n) \choose t} }\frac{M_t(\bm s, H_n)}{\varepsilon c_n^{r-t}} |\{\bm e \in E(H_n) : \bm s \subset \bm e\}| \bm 1\{ M_t(\bm s, H_n) > \varepsilon c_n^{r-t}\} \\
        &\le \sum_{t=2}^{r-1}\sum_{\bm s \in {V(H_n) \choose t} }\frac{M_t(\bm s, H_n)^2}{\varepsilon c_n^{r-t}} \bm 1\{ M_t(\bm s, H_n) > \varepsilon c_n^{r-t}\} . 
    \end{align*}
    For fixed $t, \sum_{\bm s \in V(H_n)^t}\binom{M_t(\bm s, H_n)}{2}$ counts the number of pairs of edges sharing at least $t$ vertices, up to a constant depending on $t$ and $r$. Thus, with the hypothesis that for $t \in [2, r - 1], |\cK(t, H_n)| = o(c_n^{2r-t-1})$, we have
    \begin{align}\label{eq:EHnAHn}
        |E(H_n)\setminus\cA_{\varepsilon}(H_n)| &\le\sum_{t=2}^{r-1}\sum_{\bm s \in V(H_n)^t}\frac{M_t(\bm s, H_n)^2}{\varepsilon c_n^{r-t}} \bm 1\{ M_t(\bm s, H_n) > \varepsilon c_n^{r-t}\} \nonumber \\
        & \lesssim \sum_{t=2}^{r-1}\frac{1}{\varepsilon c_n^{r-t}}\sum_{\bm s \in V(H_n)^t}\binom{M_t(\bm s, H_n)}{2} \nonumber \\
        &\lesssim_{r,t} \sum_{t=2}^{r-1}\frac{\cK(t, H_n)}{\varepsilon c_n^{r-t}} . 
    \end{align}
    Hence, the second condition of Theorem \ref{thm:THnKpoisson} implies
    \begin{align}\label{eq:Hnremainder}
    \bE T_{\varepsilon}^-(H_n) = \frac{1}{c_n^{r-1}}|E(H_n)\setminus\cA_{\varepsilon}(H_n)|\lesssim_{r,t} \sum_{t=2}^{r-1}\frac{\cK(t, H_n)}{\varepsilon c_n^{2r-t-1}} = o(1). 
    \end{align}
    This completes the proof of Lemma \ref{lm:remainderH}. 
\end{proof}

Next, we analyze the main term $T_{\varepsilon}^+(H_n)$ (recall \eqref{eq:THnepsilon}). To do so, we consider a surrogate for $T_{\varepsilon}^+(H_n)$ obtained by replacing the variables $\{ \bm 1\{X_{=\bm e}\} : \bm e \in  \cA_{\varepsilon}(H_n)\}$ by independent Bernoulli random variables. Formally, we let 
\begin{align}\label{eq:Jn}
J_{\varepsilon}^+(H_n) := \sum_{\bm e \in \cA_{\varepsilon}(H_n)}J_{\bm e} , 
\end{align}
where $\{J_{\bm e} : \bm e \in  \cA_{\varepsilon}(H_n)\}$ is a collection of independent $\Ber(\frac{1}{c_n^{r-1}})$ random variables. The following lemma shows that $T_{\varepsilon}^+(H_n)$ and $J_{\varepsilon}^+(H_n)$ are asymptotically close in moments. 

\begin{lem}
    \label{lm:momentH}
    For all integers $k \ge 1$, 
    $$ |\bE T_{\varepsilon}^+(H_n)^k - \bE J_{\varepsilon}^+(H_n)^k|  \rightarrow 0 , $$ 
    as $n \rightarrow \infty$ followed by $\varepsilon \rightarrow 0$.
\end{lem}

The proof of Lemma \ref{lm:momentH} is given in Section \ref{sec:Hnmainpf}. To see how Lemma \ref{lm:momentH} applies in completing the proof of Theorem \ref{thm:THnKpoisson}, note that 
$$\bE J_{\varepsilon}^+(H_n) = \frac{|\cA_{\varepsilon}(H_n)|}{c_n^{r-1}} = \frac{|E(H_n)|}{c_n^{r-1}} - \frac{|E(H_n)\setminus\cA_{\varepsilon}(H_n)|}{c_n^{r-1}} \rightarrow \lambda ,$$
since $\frac{1}{c_n^{r-1}}|E(H_n)| \rightarrow \lambda$ by assumption and the second term is negligible by \eqref{eq:Hnremainder}. This implies that, since $J_{\varepsilon}^+(H_n)$ is a sum of independent $\Ber(\frac{1}{c_n^{r-1}})$ random variables (recall \eqref{eq:Jn}), 
$$J_{\varepsilon}^+(H_n) \rightarrow \Pois(\lambda)$$
in distribution and in moments. Hence, by Lemma \ref{lm:momentH}, the moments of $T_{\varepsilon}^+(H_n)$ converge to the moments of $\Pois(\lambda)$. Since a Poisson distribution is uniquely determined by its moments, we obtain $T_{\varepsilon}^+(H_n) \dto  \Pois(\lambda)$. Combining this with Lemma \ref{lm:remainderH} implies $T(H_n) \dto  \Pois(\lambda)$, completing the proof of Theorem \ref{thm:THnKpoisson}.

\subsubsection{Proof of Lemma \ref{lm:momentH}} 
\label{sec:Hnmainpf} 

%Fix an integer $k \geq 1$ and let $\cH_{r, b}$ be the collection of $r$-uniform hypergraphs with at most $k$ distinct hyperedges. 

Fix integers $1 \leq b \le k,$ define
$$\sS_{\varepsilon, k, b}= \{ (\bm e_1, \ldots, \bm e_k):  \bm e_i \in \cA_{\varepsilon}(H_n), \text{ for all }  1 \leq i \leq k,  \text{ and } |\{\bm e_1, \ldots, \bm e_k\}| = b \} .$$ 
In words, $\sS_{\varepsilon, k, b}$ is the collection of $k$-tuples of hyperedges (not necessarily distinct) in $\cA_{\varepsilon}(H_n)$ such that exactly $b$ (out of the $k$) hyperedges are distinct. Given $S = (\bm e_1, \ldots, \bm e_k) \in \sS_{\varepsilon, k, b}$, we denote the $r$-uniform hypergraph formed by the union of the edges $\{\bm e_1, \ldots, \bm e_k\}$ by $\cP(S)$. Note that $\cP(S)$ has $b$ distinct hyperedges. Given an $r$-uniform hypergraph $F$, define
    $$\sS_{\varepsilon, k, b}(F) := \{ S = (\bm e_1, \ldots,\bm e_k) \in \sS_{\varepsilon, k, b}: \cP(S) \text{ is isomorphic to }  F \} . $$ 
Note that if $F$ is connected and $\sS_{\varepsilon, k, b}(F)$ is non-empty, then $|V(F)| - 1\leq rb-b$. Hence, if $F$ has $\nu(F)$ connected components, and $\sS_{\varepsilon, k, b}(F)$ is non-empty, then 
$|V(F)| -\nu(F) \leq rb-b$. The following lemma gives a bound on $\sS_{\varepsilon, k, b}(F)$: 

\begin{lem}
    \label{lm:countS} 
   Fix $r \geq 3$. Let $F= (V(F), E(F))$ be an $r$-uniform hypergraph with $\nu(F)$ connected components such that $\sS_{\varepsilon, k, b}(F)$ is non-empty.   
    Then $|\sS_{\varepsilon, k, b}(F)| \lesssim_{r, b, \lambda} c_n^{|V(F)| - \nu(F)}$. Moreover, if $|V(F)| - \nu(F) < rb-b$, then
    $$|\sS_{\varepsilon, k, b}(F)| \lesssim_{r,b,\lambda} \varepsilon c_n^{|V(F)| - \nu(F)} . $$
\end{lem}

\begin{proof} To begin with suppose $F$ is connected and $|V(F)| - 1 < rb-b$. Then by Lemma \ref{lm:connected} any $S  \in \sS_{\varepsilon, k, b}(F)$ can be ordered as $S= (\bm e_1, \ldots,\bm e_k)$ such that the following holds:  
    
    \begin{enumerate}
        \item For all $2 \leq i \leq k$, 
        \begin{align}\label{eq:sequence} 
        t_i:= \left|\bm e_i \bigcap \left(\bigcup_{j = 1}^{i-1}\bm e_j\right) \right| \geq 1. 
        \end{align}
        That is, each hyperedge $\bm e_i$ intersects with the union of the previous hyperedges. 
                
        \item   
        For some $2 \leq i \leq k$, 
        $$ t_i:= \left|\bm e_i \bigcap \left(\bigcup_{j = 1}^{i-1}\bm e_j\right) \right| \in [2,  r]. $$
        That is, there is some hyperegde $\bm e_i$ that intersects the union of the previous hyperedges in at least 2 vertices.%, but is not completely contained in the union of the previous hyperedge.  
           \end{enumerate} 
           
Note that the hyperedge $\bm e_i$ can be chosen in at most $|E(H_n)|\lesssim_{r,k, \lambda} c_n^{r-1}$ ways when $t_i=1$, in $O_{r, k}(\varepsilon c_n^{r-t_i})$ ways when $2 \leq t_i \leq r-1$ (recall the definition of  $\cA_{\varepsilon}(H_n)$ from \eqref{eq:AepsilonHn}), and in $O_{r, k}(1)$ ways when $t_i = r$, leading to
    $$|\sS_{\varepsilon, k, b}(F)| \lesssim_{r,k, \lambda} \varepsilon c_n^{|V(F)| -1}.$$
when $|V(F)|-1 < rb-b$.

Next, suppose $F$ is connected and $|V(F)| - 1 = rb-b$.  In this case, $S = (\bm e_1, \ldots,\bm e_k)\in \sS_{\varepsilon, k, b}(F)$ satisfies the equality in \eqref{eq:sequence}, regardless of the ordering. Hence, each $\bm e_i$ can be chosen in at most $|E(H_n)|\lesssim_{r,k, \lambda} c_n^{r-t_i}$ ways and  
 $$|\sS_{\varepsilon, k, b}(F)| \lesssim_{r,k, \lambda} c_n^{|V(F)| -1}.$$ 
 
 If $F$ is disconnected, by applying the above bounds to each of the connected components the result follows. 
\end{proof}

Using the above lemma, we can now complete the proof of Lemma \ref{lm:momentH}. Hereafter, for any hypergraph $F$ we use $\nu(F)$ to denote the number of connected components of $F$. By the multinomial expansion and the union bound, we have
    \begin{align}
        |\bE T_{\varepsilon}^+(H_n)^k - \bE J_{\varepsilon}^+(H_n)^k| &\le \sum_{b=1}^{k} \sum_{S = (\bm e_1, \bm e_2, \ldots, \bm e_k) \in \sS_{\varepsilon, k, b}} \left |\bE \prod_{t=1}^k\bm 1\{X_{=\bm e_t}\} - \prod_{t=1}^k\bE J_{\bm e_t}\right | \nonumber \\
        &= \sum_{b=1}^{k} \sum_{S \in \sS_{\varepsilon, k, b}} \left |\frac{1}{c_n^{|V(\cP(S))| - |\nu(\cP(S))|}} - \frac{1}{c_n^{br-b}}\right | \nonumber \\
        &\lesssim_{r,k} \sum_{b=1}^{k} \sum_{F \in \cH_{r, b}}\left |\frac{1}{c_n^{|V(F)| - |\nu(F)|}} - \frac{1}{c_n^{br-b}}\right |\cdot |\sS_{\varepsilon, k, b}(F)| , 
        \label{eq:momentXJTHn}
    \end{align} 
    where $\cH_{r, b}$ is the collection of $r$-uniform hypergraphs with $b$ hyperedges. If $|V(F)| - |\nu(F)| = br-b$, then the corresponding term in the sum above is zero. Moreover, whenever $\sS_{\varepsilon, k, b}(F)$ is non-empty, $|V(F)| - |\nu(F)| \leq br-b$. Hence, it suffices to consider only those $F \in \cH_{r, b}$ for which $|V(F)| - |\nu(F)| < br-b$. In this case, by Lemma \ref{lm:countS}, 
    $$|\sS_{\varepsilon, k, b}(F)| \lesssim_{r,k, \lambda} \varepsilon c_n^{|V(F)| - \nu(F)}.$$
Hence,     
\begin{align*}
        |\bE T_{\varepsilon}^+(H_n)^k - \bE J_{\varepsilon}^+(H_n)^k| &\lesssim_{r,k} \sum_{b=1}^{k} \sum_{F \in \cH_{r, b}}\left |\frac{1}{c_n^{|V(F)| - |\nu(F)|}} - \frac{1}{c_n^{br-b}}\right |\cdot |\sS_{\varepsilon, k, b}(F)|\\
        &\lesssim_{r,k} \sum_{b=1}^{k} \sum_{F \in \cH_{r, b}}\varepsilon \cdot \bm 1\{|V(F)| - |\nu(F)| < br-b\} , 
    \end{align*}
    which goes to 0 as $n \rightarrow \infty$ followed by $\varepsilon \rightarrow 0$. This completes the proof of Lemma \ref{lm:momentH}. 
\hfill $\Box$

\subsubsection{Completing the Proof of Theorem \ref{thm:THnKpoisson} when $r = 2$.} 
\label{sec:THnKgraphpf}

Note that when $r=2$ the set \eqref{eq:AepsilonHn} is empty (technically, it is not even defined), and the proof in Section \ref{sec:THnKhypergraphpf} breaks down. In this case, 
defining 
$$J(H_n) := \sum_{\bm e \in E(H)}J_{\bm e} , $$ 
where $\{J_{\bm e} : \bm e \in  E(H_n) \}$ is a collection of independent $\Ber(\frac{1}{c_n})$ random variables, one can directly show the following: For all integers $k \ge 1$, 
    \begin{align}\label{eq:momentgraph}
     |\bE T(H_n)^k - \bE J(H_n)^k|  \rightarrow 0 , 
     \end{align} 
    as $n \rightarrow \infty$. The proof of \eqref{eq:momentgraph} can be found in \cite[Lemma 2.4]{bhattacharya2017universal}. Here, we sketch the argument for the sake of completeness. To this end, fix integers $1 \leq b \le k$, and define
$$\sS_{k, b}= \{ (\bm e_1, \ldots, \bm e_k):  \bm e_i \in E(H_n), \text{ for all }  1 \leq i \leq k,  \text{ and } |\{\bm e_1, \ldots, \bm e_k\}| = b \} .$$ 
In words, $\sS_{k, b}$ is the collection of $k$-tuples of edges (not necessarily distinct) in $E(H_n)$ such that exactly $b$ (out of the $k$) edges are distinct. Given $S = (\bm e_1, \ldots, \bm e_k) \in \sS_{k, b}$, we denote the graph formed by the union of the edges $\{\bm e_1, \ldots, \bm e_k\}$ by $\cP(S)$. Note that $\cP(S)$ has $b$ distinct edges. Given a graph $F$, define
    \begin{align*}%\label{eq:SkbF}
    \sS_{k, b}(F) := \{ S = (\bm e_1, \ldots,\bm e_k) \in \sS_{k, b}: \cP(S) \text{  is isomorphic to }  F \} . 
    \end{align*}
    Note that if $F$ is connected and $\sS_{k, b}(F)$ is non-empty, then $|V(F)| - 1\leq b$. Hence, if $F$ has $\nu(F)$ connected components and $\sS_{k, b}(F)$ is non-empty, then 
$|V(F)| -\nu(F) \leq b$. %The following lemma gives a bound on $\sS_{\varepsilon, k, b}(F)$:  
Then, as in \eqref{eq:momentXJTHn}, 
    \begin{align*}
        |\bE T(H_n)^k - \bE J(H_n)^k| & \lesssim_{k} \sum_{b=1}^{k} \sum_{F \in \cH_{2, b}}\left |\frac{1}{c_n^{|V(F)| - |\nu(F)|}} - \frac{1}{c_n^{b}}\right |\cdot |\sS_{k, b}(F)| , \nonumber \\ 
        & \lesssim_{k} \sum_{b=1}^{k} \sum_{\substack{ F \in \cH_{2, b} \\ |V(F)| - |\nu(F)| < b }} \frac{|\sS_{k, b}(F)|}{c_n^{|V(F)| - |\nu(F)|} }   ,  
        %\label{eq:momentTHn}
    \end{align*} 
%If $|V(F)| - |\nu(F)| = b$, then the corresponding term in the sum above is zero. Moreover, whenever $\sS_{ k, b}(F)$ is non-empty, $|V(F)| - |\nu(F)| \leq b$. Hence, it suffices to consider only those $F \in \cH_{2,k}$ for which $|V(F)| - |\nu(F)| < b$. 
Note that if $F \in \cH_{2, b}$, then $|E(F)| = b$. Hence,  $|V(F)| - |\nu(F)| < b $, implies $|E(F)| > |V(F)| - |\nu(F)|$, which means that $F$ contains a cycle. Then, from \cite[Lemma 2.3]{bhattacharya2017universal} we know that 
    $$|\sS_{k, b}(F)| \lesssim_{r,k, \lambda} = o(c_n^{|V(F)| - \nu(F)}).$$
This implies the result in \eqref{eq:momentgraph}, since $|\cH_{2, b}| \lesssim_k 1$. 

Combining \eqref{eq:momentgraph} with the fact that $J(H_n) \dto \Pois(\lambda)$, we get the result in  Theorem \ref{thm:THnKpoisson} for $r = 2$.

\begin{remark}\label{remark:poissonhypergraph} As mentioned in Remark \ref{remark:steinsmethod}, Theorem \ref{thm:THn} does not follow by applying Stein's method based on a generic dependency graph construction. For instance, consider a sequence of 3-uniform hypergraphs $H_n = (V(H_n), E(H_n))$ colored with $c_n$ colors as in \eqref{eq:cn} such that $\bE(T(H_n)) = \frac{|E(H_n)|}{c_n^2} \rightarrow \lambda$. Then a natural way to construct a dependency graph for the collection of random variables $\{ \bm 1\{\bm X_{= \bm e} \} : \bm e \in E(H_n) \}$ is to put an edge between $\bm 1\{\bm X_{= \bm e}\}$ and $\bm 1\{\bm X_{= \bm e'}\}$ whenever $\bm e \cap \bm e' \ne \emptyset$, for $\bm e, \bm e' \in E(H_n)$. 
Then, applying \cite[Theorem 15]{chatterjee2005exchangeable} one can show that $T(H_n) \dto \lambda$ whenever the following conditions hold: 
$$|\cK(1, H_n)|=o(|E(H_n)|^\frac{3}{2}) \quad \text{and} \quad |\cK(2, H_n)|=o(|E(H_n)|^\frac{3}{2}),$$ 
where $\cK(\cdot, H_n)$ is as in Definition \ref{defn:K}. 
However, this does not give the second moment phenomenon because $|\cK(1, H_n)|$ is not controlled by the variance condition (observe that in Theorem \ref{thm:THnKpoisson} the condition on $|\cK(1, H_n)|$ is not required). 
%For instance, in the wheel graph on $W_n$ on $n$-vertices,\footnote{The wheel graph $W_n$ has vertex-set $V(W_n):=\{0, 1,2,\ldots, n\}$, and edge-set $E(W_n)=\{ (0,1), (0,2), \ldots, (0,n),$ $(1,2), (2,3), \ldots, (n-1,n), (n,1)\}$.} colored with $c_n$ colors such that $\bE(T(K_3, W_n))=\frac{n}{c_n^3}\rightarrow 1$,  it is easy to check that $T(K_3, W_n) \dto \Pois(1)$, but $N(\bowtie, W_n)=\frac{n(n-1)}{2}$, that is, the above dependency graph construction does not work.  
This is because, unlike the direct moment-based approach used in this paper, the generic dependency graph construction is unable to leverage the fact that the $\Cov(\bm 1\{X_{=\bm e}\}, \bm 1\{X_{=\bm e'}\})=0$, whenever $|\bm e \cap \bm e'| = 1$. It would be interesting to see whether a more sophisticated dependency graph construction or other versions of Stein's method can be used to obtain rates of convergence for Theorem \ref{thm:THn} (and also Theorems \ref{thm:hypergraphr} and \ref{thm:r1r2}). 
\end{remark} 

\section{Proof of Theorem \ref{thm:hypergraphr} }
\label{sec:hypergraphrpf}

Let $H_n^{(1)}  = (V(H_n^{(1)}), E(H_n^{(1)}))$ and $H_n^{(2)}  = (V(H_n^{(2)}), E(H_n^{(2)}))$ be two sequences of $r$-uniform hypergraphs on the same set of vertices. We will denote the common vertex set by $V_n$, that is, $V(H_n^{(1)}) = V(H_n^{(2)}) = V_n$. It is now helpful to introduce the following  notation: 

    \begin{itemize}
        \item Denote by $H_n^{(1)} \cup H_n^{(2)}$ the hypergraph with vertex set $V_n$ and edge set $E(H_n^{(1)}) \cup E(H_n^{(2)})$.
        
    \item Denote by $H_n^{(1)} \setminus H_n^{(2)}$ the hypergraph with vertex set $V_n$ and edge set $E(H_n^{(1)}) \setminus E(H_n^{(2)})$ and $H_n^{(2)} \setminus H_n^{(1)}$ similarly.
    
        \item Denote by $H_n^{(1, 2)} := H_n^{(1)} \cap H_n^{(2)}$ the hypergraph with vertex set $V_n$ and edge set $E(H_n^{(1)}) \cap E(H_n^{(2)})$.
    \end{itemize}

Theorem \ref{thm:hypergraphr} follows from a more general result about the joint distribution of $(T(H_n^{(1)}\setminus H_n^{(2)}), T(H_n^{(2)}\setminus H_n^{(1)}), T(H_n^{(1, 2)}))$, which is stated below.

\begin{thm}
  \label{thm:jointr}
  Fix an integer $r \ge 2$. Let $\bm H_n =(H_n^{(1)}, H_n^{(2)})$ be a sequence of 2-hypermultiplexes on the same vertex set $V_n$ with $|E(H_n^{(1)})| \to \infty$ and $|E(H_n^{(2)})| \to \infty$. Suppose $V_n$ is colored uniformly at random with $c_n$ colors as in \eqref{eq:cn}, such that the following hold:
    \begin{itemize}
      \item  There exists constants $\lambda_1, \lambda_2, \lambda_{1,2} \ge 0$ with $\lambda_1 \ge \lambda_{1,2}, \lambda_2 \ge \lambda_{1,2}$ such that 
            \begin{align}\label{eq:EH1H2r}
            \frac{|E(H_n^{(1)})|}{c_n^{r-1}} \to \lambda_1,~\frac{|E(H_n^{(2)})|}{c_n^{r-1}} \to \lambda_2, \text{ and } \frac{|E(H_n^{(1, 2)})|}{c_n^{r-1}} \to \lambda_{1,2}. 
            \end{align}

      \item  For all $t \in [2, r - 1]$
      \begin{align}\label{eq:KH1H2hyperedger}
      \frac{|\cK(t, H_n^{(1)})|}{c_n^{2r-t-1}} \to 0 \text{ and } \frac{|\cK(t, H_n^{(2)})|}{c_n^{2r-t-1}} \to 0 , 
      \end{align}
     where $\cK(t, \cdot)$ is defined in \eqref{defn:K}. 
  \end{itemize} 
Then 
$$\begin{pmatrix}
    T(H_n^{(1)}\setminus H_n^{(2)})\\
    T(H_n^{(2)}\setminus H_n^{(1)})\\
    T(H_n^{(1, 2)})
\end{pmatrix} \dto \begin{pmatrix}
        Z_1 \\
        Z_2 \\
        Z_{1,2}
        \end{pmatrix} , $$
  where $Z_1 \sim \Pois(\lambda_1 - \lambda_{1,2})$, $Z_2 \sim \Pois(\lambda_2 - \lambda_{1,2})$, and $Z_{1,2} \sim \Pois(\lambda_{1,2})$ are independent. 
\end{thm}

The proof of Theorem \ref{thm:jointr} is given in Section \ref{sec:jointrpf}. We now complete the proof of Theorem \ref{thm:hypergraphr} by showing that that the conditions of Theorem \ref{thm:hypergraphr} imply the conditions of Theorem \ref{thm:jointr}. First note that the assumptions $\lim_{n \to \infty}\Var[T(H_n^{(1)})] = \lambda_1$ and 
$ \lim_{n \to \infty}\Var[T(H_n^{(2)})] = \lambda_2$ imply, from arguments as in \eqref{eq:varianceTHn}, that for all $t \in [2, r - 1]$,      
\begin{align}\label{eq:KH1H2}
\frac{|\cK(t, H_n^{(1)})|}{c_n^{2r-t-1}} \to 0 \text{ and } \frac{|\cK(t, H_n^{(2)})|}{c_n^{2r-t-1}} \to 0 . 
\end{align} 
Hence, it remains to show that 
    \begin{align}\label{eq:EH12}
    \lim_{n \to \infty}\Cov[T(H_n^{(1)}), T(H_n^{(2)})] = \lambda_{1,2} \text{ implies } 
    \frac{|E(H_n^{(1, 2)})|}{c_n^{r-1}} \to \lambda_{1,2} .  
    \end{align}
    For this note that for two edges $\bm e_1, \bm e_2 \in E(H_n^{(1)}) \cup E(H_n^{(2)})$ such that $\bm e_1\cap \bm e_2 \ne \emptyset$, 
    $$\Cov[ {\bf 1}\{X_{=\bm e_1}\}, {\bf 1}\{X_{=\bm e_2}\} ] = \frac{1}{c_n^{2r - |\bm e_1\cap \bm e_2| - 1}} - \frac{1}{c_n^{2r-2}} , $$
    which is nonzero if and only if $2 \le |\bm e_1\cap \bm e_2| \le r$. Hence, 
    \begin{align}\label{eq:H1H2Q}
    \Cov[T(H_n^{(1)}), T(H_n^{(2)})] = Q_{1,n} + Q_{2,n} , 
    \end{align}
    where 
    $$Q_{1,n} = \frac{1}{c_n^{r-1}}\left(1 - \frac{1}{c_n^{r-1}}\right) |E(H_n^{(1, 2)})|$$
    and 
    $$Q_{2,n} \le \sum_{t=2}^{r-1}\frac{1}{c_n^{2r-t-1}}\left(1 - \frac{1}{c_n^{t-1}}\right)|\cK(t, H_n^{(1)}, H_n^{(2)})|, $$ 
      where $\cK(t, H_n^{(1)}, H_n^{(2)})$ is the set of ordered pairs of hyperedges $(\bm e_1, \bm e_2)$ such that $\bm e_1 \in E(H_n^{(1)}), \bm e_2 \in E(H_n^{(2)})$ and $|\bm e_1 \cap \bm e_2| = t$. Note that for $t\in [2, r-1]$, 
    $$|\cK(t, H_n^{(1)}, H_n^{(2)})| \le |\cK(t, H_n^{(1)})| + |\cK(t, H_n^{(2)})| + |\cK(t, H_n^{(1)} \setminus H_n^{(2)}, H_n^{(2)} \setminus H_n^{(1)} )|. $$
    As $|\cK(t, H_n^{(1)})| = o(c_n^{2r-t-1})$ and $|\cK(t, H_n^{(2)})| = o(c_n^{2r-t-1})$ (recall \eqref{eq:KH1H2}), it suffices to bound $|\cK(t, H_n^{(1)} \setminus H_n^{(2)}, H_n^{(2)} \setminus H_n^{(1)} )|$.  To this end, note that 
    \begin{align}
        |\cK(t, H_n^{(1)} \setminus H_n^{(2)}, H_n^{(2)} \setminus H_n^{(1)} )| &\le \sum_{\bm s \in V(H_n)^t}M_t(\bm s, H_n^{(1)})\cdot M_t(\bm s, H_n^{(2)}) \nonumber \\
        &\lesssim \sum_{\bm s \in V(H_n)^t}\left\{ M_t(\bm s, H_n^{(1)})^2 + M_t(\bm s, H_n^{(2)})^2\right\} \nonumber \\ 
        &\lesssim_{r, t} \left\{ \cK_t(\bm s, H_n^{(1)})  + \cK_t(\bm s, H_n^{(2)}) \right\} \tag*{(by \eqref{eq:EHnAHn})} \nonumber \\ 
        & = o(c_n^{2r-t-1}) , 
        \label{eq:KH1H2joint}
    \end{align}
    where the last step uses \eqref{eq:KH1H2}. This shows $Q_{2, n} \rightarrow 0$. Then, since $\lim_{n \to \infty}\Cov[T(H_n^{(1)}), T(H_n^{(2)})] = \lambda_{1, 2}$ by assumption, from \eqref{eq:H1H2Q} we must have $Q_{1, n} \rightarrow  \lambda_{1, 2}$. This implies $\frac{|E(H_n^{(1, 2)})|}{c_n^{r-1}} \to \lambda_{1,2}$ and establishes \eqref{eq:EH12}. Now, we can apply Theorem \ref{thm:jointr} and the continuous mapping theorem to get, 
    $$ \begin{pmatrix}
    T(H_n^{(1)}) \\ 
    T(H_n^{(2)}) 
  \end{pmatrix} = 
    \begin{pmatrix}
    T(H_n^{(1)}\setminus H_n^{(2)}) + T(H_n^{(1, 2)})\\
    T(H_n^{(2)}\setminus H_n^{(1)}) + T(H_n^{(1, 2)}) 
\end{pmatrix} \dto \begin{pmatrix}
        Z_1 + Z_{1,2} \\
        Z_2 + Z_{1,2} 
    \end{pmatrix} , $$ 
    where $Z_1, Z_2, Z_{1, 2}$ are as defined in Theorem \ref{thm:jointr}. This completes the proof of Theorem \ref{thm:hypergraphr}. \hfill $\Box$

\subsection{Proof of Theorem \ref{thm:jointr}} 
\label{sec:jointrpf}

We will present the proof for $r \geq 3$. For $r=2$, the proof needs to be modified as in Section \ref{sec:THnKgraphpf}.

Fix $r \geq 3$. Throughout, we will let $H_n = H_n^{(1)} \cup H_n^{(2)}$. Now, fix $\varepsilon > 0$ and define
$$ \bm T : = \begin{pmatrix}
    T(H_n^{(1)}\setminus H_n^{(2)})\\
    T(H_n^{(2)}\setminus H_n^{(1)})\\
    T(H_n^{(1, 2)})
\end{pmatrix}  = \begin{pmatrix}
    T_{\varepsilon}^+(H_n^{(1)}\setminus H_n^{(2)})\\
    T_{\varepsilon}^+(H_n^{(2)}\setminus H_n^{(1)})\\
    T_{\varepsilon}^+(H_n^{(1, 2)})
\end{pmatrix} + \begin{pmatrix}
    T_{\varepsilon}^-(H_n^{(1)}\setminus H_n^{(2)})\\
    T_{\varepsilon}^-(H_n^{(2)}\setminus H_n^{(1)})\\
    T_{\varepsilon}^-(H_n^{(1, 2)})  
\end{pmatrix} := \bm{T}_{n, \varepsilon}^+ + \bm{T}_{n, \varepsilon}^- , $$ 
where %(recalling that $H_n = H_n^{(1)} \cup H_n^{(2)}$)
\begin{align*}
    T_{\varepsilon}^+(F) &= \sum_{ \bm e \in E(F)}\bm 1\{X_{=\bm e}\} \bm 1\{ \bm e \in \cA_{\varepsilon}(H_n) \} \\
    \end{align*} 
    and 
    \begin{align*}
     T_{\varepsilon}^-(F) &= \sum_{ \bm e \in E(F)}\bm 1\{X_{=\bm e}\} \bm 1\{ \bm e \notin \cA_{\varepsilon}(H_n) \} , 
     \end{align*}
     for $F = \{H_n^{(1)}\setminus H_n^{(2)}, H_n^{(2)}\setminus H_n^{(1)}, H_n^{(1, 2)}\}$. 
Notice that 
\begin{align}\label{eq:TremainderH1H2}
T_{\varepsilon}^-(H_n) & = \sum_{ \bm e \in E(H_n) }\bm 1\{X_{=\bm e}\} \bm 1\{ \bm e \notin \cA_{\varepsilon}(H_n) \}  \nonumber \\ 
& = T_{\varepsilon}^-(H_n^{(1)}\setminus H_n^{(2)}) + T_{\varepsilon}^-(H_n^{(2)}\setminus H_n^{(1)})+ T_{\varepsilon}^-(H_n^{(1, 2)}). 
\end{align}
Since $|\cK(t, H_n)| = o(c^{2r-t-1})$, for all $t \in [2, r-1]$ (recall \eqref{eq:KH1H2} and \eqref{eq:KH1H2joint}), by arguments as in Lemma \ref{lm:remainderH}, for each fixed $\varepsilon > 0$, $T_{\varepsilon}^-(H_n) \overset{L_1}{\to} 0$, as $n \to \infty$. This implies, from \eqref{eq:TremainderH1H2}, $\bm T_{n, \varepsilon}^{-} \overset{L_1}{\to} 0$. Hence, it suffices to analyze $\bm T_{n, \varepsilon}^{+}$. 
As in Section \ref{sec:poissonmarginalpf}, we will approximate the variables $\{ \bm 1\{X_{=\bm e}\} : \bm e \in  \cA_{\varepsilon}(H_n)\}$ by independent Bernoulli random variables. In particular, define 
\begin{align*}
    J_{\varepsilon}^+(F) &= \sum_{\bm e \in E(F) \cap \cA_{\varepsilon}(H_n) } J_{\bm e}, \quad \text{ for } F = \{H_n^{(1)}\setminus H_n^{(2)}, H_n^{(2)}\setminus H_n^{(1)}, H_n^{(1, 2)}\} ,  
\end{align*} 
where $\{J_{\bm e} : \bm e \in  \cA_{\varepsilon}(H_n)\}$ is a collection of independent $\Ber(\frac{1}{c_n^{r-1}})$ random variables. 
This implies,  
$$\bE J_{\varepsilon}^+(F) = \frac{|E(F) \cap \cA_{\varepsilon}(H_n)|}{c_n^{r-1}} =  \frac{|E(F)|}{c_n^{r-1}} - \frac{|E(F) \backslash \cA_{\varepsilon}(H_n)|}{c_n^{r-1}}.$$
Note that, for $F = \{H_n^{(1)}\setminus H_n^{(2)}, H_n^{(2)}\setminus H_n^{(1)}, H_n^{(1, 2)}\}$, 
$$\frac{|E(F) \backslash \cA_{\varepsilon}(H_n)|}{c_n^{r-1}} \leq \frac{|E(H_n) \backslash \cA_{\varepsilon}(H_n)|}{c_n^{r-1}} \rightarrow 0 , $$ 
by \eqref{eq:Hnremainder}. 
Hence, because $\bm J_{n, \varepsilon}^+$ is a sum of independent random vectors, 
$$\bm J_{n, \varepsilon}^+ : = 
\begin{pmatrix}
 J_{\varepsilon}^+(H_n^{(1)}\setminus H_n^{(2)})\\
 J_{\varepsilon}^+(H_n^{(2)}\setminus H_n^{(1)})\\
 J_{\varepsilon}^+(H_n^{(1, 2)})
\end{pmatrix}  \dto 
\begin{pmatrix}
        Z_1 \\
        Z_2 \\
        Z_{1,2}
        \end{pmatrix} := \bm Z , 
$$ 
where $Z_1, Z_2, Z_{1, 2}$ are as defined in Theorem \ref{thm:jointr}. 
The result in Theorem \ref{thm:jointr} now follows from the next lemma. 

\begin{lem}\label{lm:momentH1H2}
For all integers $k, \bar k, \underline k \ge 0$, 
    $$|\bE[T_{\varepsilon}^+(H_n^{(1)}\setminus H_n^{(2)})^k T_{\varepsilon}^+(H_n^{(2)}\setminus H_n^{(1)})^{\bar k} T_{\varepsilon}^+(H_n^{(1, 2)})^{\underline k} ] - \bE[J_{\varepsilon}^+(H_n^{(1)}\setminus H_n^{(2)})^k J_{\varepsilon}^+(H_n^{(2)}\setminus H_n^{(1)})^{\bar k} J_{\varepsilon}^+(H_n^{(1, 2)})^{\underline k} ] | \rightarrow 0 , $$ 
    as $n \rightarrow \infty$ followed by $\varepsilon \rightarrow 0$. 
\end{lem}

\begin{proof} 
Note that by the multinomial expansion, 
\begin{align*}
\bE[T_{\varepsilon}^+(H_n^{(1)}\setminus H_n^{(2)})^k T_{\varepsilon}^+(H_n^{(2)}\setminus H_n^{(1)})^{\bar k} T_{\varepsilon}^+(H_n^{(1, 2)})^{\underline k} ] 
& = \sum_{b = 1}^{k} \sum_{\bar b = 1}^{\bar k} \sum_{\underline b = 1}^{\underline k} \sum_{S \in \sS_{\varepsilon, k + \bar k + \underline k , b + \bar b+ \underline b}} \frac{1}{c_n^{|V(\cP(S))| - |\nu(\cP(S))|}} \nonumber.
\end{align*} 
Similarly, 
\begin{align*}
\bE[J_{\varepsilon}^+(H_n^{(1)}\setminus H_n^{(2)})^k J_{\varepsilon}^+(H_n^{(2)}\setminus H_n^{(1)})^{\bar k} J_{\varepsilon}^+(H_n^{(1, 2)})^{\underline k} ] & = \sum_{b = 1}^{k} \sum_{\bar b = 1}^{\bar k} \sum_{\underline b = 1}^{\underline k} \sum_{S \in \sS_{\varepsilon, k + \bar k + \underline k , b + \bar b+ \underline b}} \frac{1}{c_n^{(b + \bar b+ \underline b) r - (b + \bar b+ \underline b)}}
\end{align*} 
Grouping the terms, we obtain
\begin{align*}
& \left|\bE[T_{\varepsilon}^+(H_n^{(1)}\setminus H_n^{(2)})^k T_{\varepsilon}^+(H_n^{(2)}\setminus H_n^{(1)})^{\bar k} T_{\varepsilon}^+(H_n^{(1, 2)})^{\underline k} ]  - \bE[J_{\varepsilon}^+(H_n^{(1)}\setminus H_n^{(2)})^k J_{\varepsilon}^+(H_n^{(2)}\setminus H_n^{(1)})^{\bar k} J_{\varepsilon}^+(H_n^{(1, 2)})^{\underline k} ] \right| \\
\le & \sum_{b = 1}^{k} \sum_{\bar b = 1}^{\bar k} \sum_{\underline b = 1}^{\underline k} \sum_{S \in \sS_{\varepsilon, k + \bar k + \underline k , b + \bar b+ \underline b}} \left |\frac{1}{c_n^{|V(\cP(S))| - |\nu(\cP(S))|}} - \frac{1}{c_n^{(b + \bar b+ \underline b) r - (b + \bar b+ \underline b)}}\right|\\
\lesssim & \sum_{B = 1}^{k+\bar k+\underline k} \sum_{S \in \sS_{\varepsilon, k + \bar k + \underline k , B}} \left|\frac{1}{c_n^{|V(\cP(S))| - |\nu(\cP(S))|}} - \frac{1}{c_n^{B r - B}} \right|.  
\end{align*} 

The result in Lemma \ref{lm:momentH1H2} now follows by applying the arguments in the proof of Lemma \ref{lm:momentH}. 
\end{proof}

\section{Proof of Theorem \ref{thm:r1r2} }
\label{sec:r1r2pf}

%
%In this section, we see that the dependence between $T(H_n^{(1)})$ and $T(H_n^{(2)})$ disappears in the limit when $r_1 < r_2$. %The main reason is that when the second moment phenomenon holds true for the two marginal distributions, the ``intersection" will not arise. 
%We will first present the routine proof for the following theorem and we will explain the disappearance of the ``intersection" in further details in Proposition \ref{prop: conditions violated}.
%

As in the previous sections, the following theorem will imply the result in Theorem \ref{thm:r1r2}.

\begin{thm}
  \label{thm:jointr1r2}
 Fix integers $2 \le r_1 < r_2$. Let $\bm H_n = (H_n^{(1)}, H_n^{(2)})$ be a sequence of 2-hypermultiplexes on the same set of vertices $V_n$, such that 
$H_n^{(1)}$ is $r_1$-uniform, $H_n^{(2)}$ is $r_2$-uniform, $|E(H_n^{(1)})| \to \infty,$ and $|E(H_n^{(2)})| \to \infty$. Suppose $V_n$ is colored uniformly at random with $c_n$ colors as in \eqref{eq:cn}, such that the following hold: 
      \begin{itemize}
      \item  There exist constants $\lambda_1, \lambda_2 \ge 0$ such that 
            \begin{align}\label{eq:EH1H2r1r2}
            \frac{|E(H_n^{(1)})|}{c_n^{r_1-1}} \to \lambda_1,~\frac{|E(H_n^{(2)})|}{c_n^{r_2-1}} \to \lambda_2 . 
            \end{align}
        
      \item  For all $t_1 \in [2, r_1 - 1]$ and $t_2 \in [2, r_2 - 1]$, 
      \begin{align}\label{eq:KH1H2hyperedger1r2}
      \frac{|\cK(t_1, H_n^{(1)})|}{c_n^{2r_1-t_1-1}} \to 0 \text{ and } \frac{|\cK(t_2, H_n^{(2)})|}{c_n^{2r_2-t_2-1}} \to 0 , 
      \end{align} 
     where $\cK(\cdot, \cdot)$ is defined in \eqref{defn:K}. 
  \end{itemize} 
Then 
  $$\bm T(\bm H_n)= \left(\begin{array}{l}
    T(H_n^{(1)}) \\
    T(H_n^{(2)})
    \end{array}\right) \dto \left(\begin{array}{l}
      Z_1\\
      Z_2
      \end{array}\right) , $$
    where $Z_1 \sim \Pois(\lambda_1), Z_2 \sim \Pois(\lambda_2)$ are independent. 
\end{thm}

Note that by arguments as in \eqref{eq:varianceTHn} it follows that the conditions of Theorem \ref{thm:r1r2} imply the conditions of Theorem \ref{thm:jointr1r2}. Hence, the result in Theorem \ref{thm:r1r2} follows from Theorem \ref{thm:jointr1r2}.

\subsection{Proof of Theorem \ref{thm:jointr1r2}} 
\label{sec:jointr1r2pf}

We will present the proof for $r_1 \geq 3$. For $r_1=2$, the proof needs to be modified as in Section \ref{sec:THnKgraphpf}.

Fix $3 \leq r_1 < r_2$, $t_1 \in [2, r_1-1]$, $t_2 \in [2, r_2-1]$,  and define 
$M_{t_1}(\bm s, H_n^{(1)})$ and  $M_{t_2}(\bm s, H_n^{(2)})$ as in \eqref{eq:MsHn} (with $H_n$ replaced by $H_n^{(1)}$ and $H_n^{(2)}$, respectively). Also, for every fixed $\varepsilon > 0$, define $\cA_{\varepsilon}(H_n^{(1)}) $ and $\cA_{\varepsilon}(H_n^{(2)}) $ as in \eqref{eq:AepsilonHn}.  Then define the {\it main terms} as: 
$$T_{\varepsilon}^+(H_n^{(i)}) = \sum_{\bm e \in  \cA_{\varepsilon}(H_n^{(i)})}\bm 1\{X_{=\bm e}\}$$
and the {\it remainder terms} as 
        $$T_{\varepsilon}^-(H_n^{(i)}) = T(H_n^{(i)}) - T_{\varepsilon}^+(H_n^{(i)}),$$
for $i \in \{1, 2\}$. By Lemma \ref{lm:remainderH}, for each fixed $\varepsilon > 0$, $T_{\varepsilon}^-(H_n^{(i)}) \overset{L_1}{\to} 0$, as $n \to \infty$, for $i \in \{1,2\}$. Thus, it suffices to analyze the joint distribution of the main terms:  
$$\bm T_{n, \varepsilon}^+  := \begin{pmatrix}
    T_{\varepsilon}^+(H_n^{(1)})\\
    T_{\varepsilon}^+(H_n^{(2)})
\end{pmatrix}.$$ 
As before, we will approximate the variables $\{ \bm 1\{X_{=\bm e}\} : \bm e \in  \cA_{\varepsilon}(H_n)\}$ by independent Bernoulli random variables. In particular, define 
\begin{align*}
    J_{\varepsilon}^+(H^{(i)}) &= \sum_{\bm e \in \cA_{\varepsilon}(H_n^{{(i)}}) } J_{\bm e} ,  
\end{align*}   
where $\{J_{\bm e} : \bm e \in  \cA_{\varepsilon}(H_n^{(i)})\}$ is a collection of independent $\Ber(\frac{1}{c_n^{r_i-1}})$ random variables, which are also independent over $i \in \{1,2\}$. Clearly, 
$$\bm J_{n, \varepsilon}^+  := \begin{pmatrix}
    J_{\varepsilon}^+(H_n^{(1)})\\
    J_{\varepsilon}^+(H_n^{(2)})
\end{pmatrix} \dto \left(\begin{array}{l}
      Z_1\\
      Z_2
      \end{array}\right) = \bm Z, $$
    where $Z_1, Z_2$ are as defined in Theorem \ref{thm:jointr1r2}. 
The result in Theorem \ref{thm:jointr} then follows from the next lemma.

\begin{lem}\label{lm:momentH1r1H2r2}
For all integers $k_1, k_2 \ge 0$, 
    $$ |\bE[T_{\varepsilon}^+(H_n^{(1)})^{k_1} T_{\varepsilon}^+(H_n^{(2)})^{k_2}] - \bE[J_{\varepsilon}^+(H_n^{(1)})^{k_1} J_{\varepsilon}^+(H_n^{(2)})^{k_2}]| \rightarrow 0 , $$ 
    as $n \rightarrow \infty$ followed by $\varepsilon \rightarrow 0$. 
\end{lem}

\subsubsection{Proof of Lemma \ref{lm:momentH1r1H2r2} } 
\label{sec:momentH1r1H2r2pf}

Fix integers $0 \leq b_1 \le k_1$ and $0 \leq b_2 \leq k_2$ and denote by 
\begin{align} 
\sS_{\varepsilon, k_1, k_2, b_1, b_2} = \Big\{ ( (\bm e_1^{(1)}, \bm e_2^{(1)}, \ldots, \bm e_{k_1}^{(1)}), & (\bm e_1^{(2)}, \bm e_2^{(2)} \ldots, \bm e_{k_2}^{(2)}) ) :  \bm e_j^{(i)} \in \cA_{\varepsilon}(H_n^{(i)}), \text{ for }  1 \leq j \leq k_i ,  \nonumber \\ 
& \text{ and } |\{\bm e_j^{(i)}: 1 \leq j \leq k_i\}| = b_i, \text{ for } i \in \{1, 2\} \Big\} . \nonumber 
\end{align}  
Given a pair $(S_1, S_2) \in \sS_{\varepsilon, k_1, k_2, b_1, b_2}$, with $S_1 = (\bm e_1^{(1)}, \bm e_2^{(1)}, \ldots, \bm e_{k_1}^{(1)})$ and $S_2 = (\bm e_1^{(2)}, \bm e_2^{(2)}, \ldots, \bm e_{k_2}^{(2)})$,  denote the hypergraph formed by the union of the edges $$\left\{\bm e_1^{(1)}, \bm e_2^{(1)}, \ldots, \bm e_{k_1}^{(1)}, \bm e_1^{(2)}, \bm e_2^{(2)}, \ldots, \bm e_{k_2}^{(2)}\right\}$$ by $\cP(S_1, S_2)$. (Note that $\cP(S_1, S_2)$ has $b_1+b_2$ distinct hyperedges.) Moreover, given a hypergraph $F$ define
        $$\sS_{\varepsilon, k_1, k_2, b_1, b_2}(F) := \{ (S_1, S_2) \in \sS_{\varepsilon, k_1, k_2, b_1, b_2}: \cP(S_1, S_2) \text{ is isomorphic to }  F \} .$$ 
        Note that if $F$ is connected and $\sS_{\varepsilon, k_1, k_2, b_1, b_2}(F)$ is non-empty then $|V(F)| \leq r_1b_1-b_1 + r_2 b_2 - b_2 +1$. Hence, if $F$ has $\nu(F)$ connected components and $\sS_{\varepsilon, k_1, k_2, b_1, b_2}(F)$ is non-empty, 
$|V(F)| -\nu(F) \leq r_1b_1 - b_1 + r_2 b_2 - b_2 $. The following lemma gives a bound on $\sS_{\varepsilon, k_1, k_2, b_1, b_2}(F)$. The proof is similar to Lemma \ref{lm:countS}. The details are omitted.

\begin{lem} 
Let $F= (V(F), E(F))$ be a hypergraph with $\nu(F)$ connected components such that $\sS_{\varepsilon, k, b}(F)$ is non-empty.  Then, $|\sS_{\varepsilon, k_1, k_2, b_1, b_2}(F)| \lesssim_{r_1, r_2, b_1, b_2, \lambda_1, \lambda_2} c_n^{|V(F)| - \nu(F)}$.  
  Moreover, if $|V(F)| < r_1b_1-b_1+r_2b_2-b_2+1$, then
    $$|\sS_{\varepsilon, k_1, k_2, b_1, b_2}(F)| \lesssim_{r_1, r_2, b_1, b_2, \lambda_1, \lambda_2} \varepsilon c_n^{|V(F)| - \nu(F)} . $$
\end{lem}

We can now apply the above lemma to complete the proof of Lemma \ref{lm:momentH1r1H2r2}. To begin with, let $\cH_{r_1, r_2, k_1, k_2}$ be the collection of hypergraphs $F = (V(F), E(F))$ that can be expressed as $F = F_1 \cup F_2$, with $V(F) = V(F_1) \cup V(F_2)$ and $E(F) = E(F_1) \cup E(F_2)$, such that $F_1$ is $r_1$-uniform and has at most $k_1$ hyperedges and $F_2$ is $r_2$-uniform and has at most $k_2$ hyperedges. Then by the multinomial expansion, we have 
\begin{equation*}
    \begin{split}
    &|\bE [ T_{\varepsilon}^+(H_n^{(1)})^{k_1} T_{\varepsilon}^+(H_n^{(2)})^{k_2} ] - \bE [ J_{\varepsilon}^+(H_n^{(1)})^{k_1} J_{\varepsilon}^+(H_n^{(2)})^{k_2} ] |\\
    &\le \sum_{b_1 = 1}^{k_1}\sum_{b_2 = 1}^{k_2}\sum_{(S_1, S_2) \in \sS_{\varepsilon, k_1, k_2, b_1, b_2}} \left|\bE \left[ \prod_{t=1}^{k_1}  \bm 1\{X_{= \bm e_t^{(1)}}\} \prod_{t=1}^{k_2}  \bm 1\{X_{= \bm e_t^{(2)}}\} \right]  - \bE \left[ \prod_{t=1}^{k_1}  J_{\bm e_t^{(1)} } \prod_{t=1}^{k_2}  J_{\bm e_t^{(2)} } \right] \right|\\
    &= \sum_{b_1 = 1}^{k_1}\sum_{b_2 = 1}^{k_2}\sum_{(S_1, S_2) \in \sS_{\varepsilon, k_1, k_2, b_1, b_2}} \left |\frac{1}{c_n^{|V(\cP(S_1, S_2))| - |\nu(\cP(S_1, S_2))|}} - \frac{1}{c_n^{r_1b_1 - b_1 + r_2b_2 - b_2}}\right |\\
    &\lesssim_{r_1, r_2, k_1,k_2} \sum_{b_1 = 1}^{k_1}\sum_{b_2 = 1}^{k_2} \sum_{F \in \cH_{r_1, r_2, k_1, k_2}}\left |\frac{1}{c_n^{|V(F)| - |\nu(F)|}} - \frac{1}{c_n^{r_1b_1 - b_1 + r_2b_2 - b_2}}\right |\cdot |\sS_{\varepsilon, k_1, k_2, b_1, b_2}(F)| \\ 
    & \lesssim_{r_1, r_2, k_1,k_2} \sum_{b_1 = 1}^{k_1}\sum_{b_2 = 1}^{k_2} \sum_{F \in \cH_{r_1, r_2, k_1, k_2}} \varepsilon \cdot \bm 1\{|V(F)| - |\nu(F)| < r_1 b_1 - b_1 + r_2 b_2 - b_2 +1 \} . 
    \end{split}
    \end{equation*}
    which goes to 0 as $n \rightarrow \infty$ followed by $\varepsilon \rightarrow 0$. This completes the proof of Lemma \ref{lm:momentH}. 
\hfill $\Box$

\begin{remark}
It might be a bit puzzling at first glance to spot the differences between the proofs of Theorem \ref{thm:jointr1r2} (where $r_1 < r_2$) and  Theorem \ref{thm:jointr} (where $r_1 = r_2 = r$). 
The subtlety lies in the asymptotic behavior $|E(H_n^{(1, 2)})|$, number of hyperedges in intersection of $H_n^{(1)}$ and $H_n^{(2)}$. Note that when $r_1 = r_2 = r$, then $|E(H_n^{(1, 2)})| = \Theta(c_n^{r-1})$ (if $\lambda_{12} > 0$ in \eqref{eq:EH1H2r}). Hence, when $r_1 = r_2 = r$, the monochromatic hyperedges from $H^{(1, 2)}$ can have a non-trivial contribution to the limiting distribution. On the other hand, when $r_1 < r_2$, then the hypergraph $H^{(1, 2)}$ is empty. Also, the number of pairs of edges in $E(H_n^{(1)}) \cup E(H_n^{(2)})$ which share at least $r_1$ vertices is at most 
  $$|\cK(r_1, H_n^{(1)})| + |\cK(r_1, H_n^{(2)})| + |\cK(r_1, H_n^{(1)}, H_n^{(2)})| . $$
  Note that $|\cK(r_1, H_n^{(1)})| = 0$ and $|\cK(r_1, H_n^{(2)})| = o(c_n^{2r_2-r_1-1})$ (recall \eqref{eq:KH1H2hyperedger1r2}). Also, since each hyperedge in $H_n^{(1)}$ has exactly $r_1$ vertices, by the same argument as in \eqref{eq:KH1H2joint},
  $$|\cK(r_1, H_n^{(1)}, H_n^{(2)})| = o(c_n^{2r_2-r_1-1}) . $$ 
  Hence, when $r_1 < r_2$ the number of monochromatic hyperedges in $H_n^{(1)}$ that are completely contained in some monochromatic hyperedge in $H_n^{(2)}$ have a negligible contribution in the limit. As a result, the joint distribution of $T(H_n^{(1)})$ and $T(H_n^{(2)})$ are asymptotically independent, when $r_1 < r_2$. 
\end{remark}

\section{Proof of Proposition \ref{ppn:dmultiplegraphs}}
\label{sec:layersjointpf}

We begin with the following proposition, which follows by a straightforward extension of the proof of Theorem \ref{thm:jointr1r2} from the bivariate case to the multivariate case.

\begin{prop}\label{ppn:dhypergraphs}
    Fix $d \geq 1$ and integers $2 \leq r_1 \leq r_2 \cdots \leq r_d$. Let $\bm H_n = (H_n^{(1)}, \ldots, H_n^{(d)})$ be a sequence of $d$-hypermultiplexes, where $H_n^{(i)}$ is $r_i$-uniform, for $1 \leq i \leq d$. Suppose the vertices of $\bm{H}_n$ are colored uniformly at random with $c_n$ colors such that the following hold:  
    
\begin{enumerate}

        \item For every $1 \leq i \leq d$, $\frac{|E(H_n^{(i)})|}{c_n^{r_i-1}} \to \lambda_i$,  
        for some constant $\lambda_i \geq 0$.

        \item  For all $1 \leq i \leq d$ and $t \in [2, r_i - 1]$, $\frac{|\cK(t, H_n^{(i)})|}{c_n^{2r_i-t-1}} \to 0$.

\item For $1 \leq i \ne j \leq d$, $E(H_n^{(i)}) \cap E(H_n^{(j)}) = \emptyset$. 
    \end{enumerate}
    Then, as $n \rightarrow \infty$, 
    $$ \bm T(\bm H_n) \dto \bm Z, $$
            where $\bm T(\bm H_n) $ and $\bm Z$ are as defined in Proposition \ref{ppn:dmultiplegraphs}.  \hfill $\Box$
\end{prop}

Proposition \ref{ppn:dmultiplegraphs} is the consequence of the previous result. To see this note that $\mathrm{Var}(H_n^{(i)}) \to \lambda_i$, for $1 \leq i \leq d$ imply conditions (1) and (2) in Proposition \ref{ppn:dhypergraphs} (recall the discussion after \eqref{eq:varianceTHn}). Proposition \ref{ppn:dmultiplegraphs} hence follows from Proposition \ref{ppn:dhypergraphs}.

\section{Proofs for Section \ref{sec:applications} } 
\label{sec:applicationspf}

In this section we collect the proofs of the results from Section \ref{sec:applications}. The section is organized as follows: We prove Corollary \ref{cor:jointedge}, Corollary \ref{cor:apn}, Corollary \ref{cor:randomhypergraph}, and Corollary \ref{cor:WHnlambda} in 
Section \ref{sec:jointedgepf}, Section \ref{sec:apnpf}, Section \ref{sec:randomhypergraphpf}, and Section \ref{sec:WHnlambdapf}, respectively.

\subsection{Proof of Corollary \ref{cor:jointedge} }
\label{sec:jointedgepf}

Let $F_1, F_2, \ldots, F_d$ be a collection of finite graphs as in  Corollary \ref{cor:jointedge}. Note that $R(F_i, G_n) = T_n(H_{G_n}(F_i))$, where $H_{G_n}(F_i)$ is the hypergraph defined in Example \ref{example:edgecoloring}. Note that $H_{G_n}(F_i)$ is a $|E(F_i)|$-uniform hypergraph. Hence, for $1 \leq i < j \leq d$, if $|E(F_i)| \ne |E(F_j)|$, then $$|E(H_{G_n}(F_i)) \cap E(H_{G_n}(F_j))| = \emptyset.$$ Moreover, when $|E(F_i)| = |E(F_j)|$, then we also have $$|E(H_{G_n}(F_i)) \cap E(H_{G_n}(F_j)) | = \emptyset , $$ since $F_i$ and $F_j$ being non-isomorphic means that the edge sets of $F_i$ are $F_j$ are different. This shows that the hyperedge sets of $T_n(H_{G_n}(F_1)), T_n(H_{G_n}(F_2)), \ldots, T_n(H_{G_n}(F_d))$ are mutually disjoint. Hence,  Proposition \ref{ppn:dmultiplegraphs} implies the result in Corollary \ref{cor:jointedge}. \hfill $\Box$

\subsection{Proof of Corollary \ref{cor:apn} } 
\label{sec:apnpf}

We begin by computing the expectation of $W_r([n])$. A direct calculation shows that 
\begin{align}\label{eq:apnr}
N_r([n]) = (1+o(1))\frac{1}{r-1} n^2. 
\end{align} 
Hence, from \eqref{eq:Anr} recalling that $c_n = \lambda n^{\frac{2}{r-1}} $ we have 
$$\bE W_r([n]) = \frac{N_r([n])}{c_n^{r-1}} \rightarrow \frac{\lambda}{r-1} . $$

Next, we compute the variance of $W_r([n])$. 
Note that $W_r([n]) = T(H_{[n]}(r))$, where $H_{[n]}(r)$ is the $r$-uniform hypergraph defined in Example \ref{example:ap}. For this, from \eqref{eq:varianceTHn}, we have 
$$\Var T(H_{[n]}(r) = R_{1,n} + R_{2,n}, $$ 
where 
\begin{align*}
R_{1,n} = \frac{1}{c_n^{r-1}}\left(1 - \frac{1}{c_n^{r-1}}\right)|E(T(H_{[n]}(r))) | \text{ and } R_{2,n} = \sum_{t=2}^{r-1}\frac{1}{c_n^{2r-t-1}}\left(1 - \frac{1}{c_n^{t-1}}\right)|\cK(t, T(H_{[n]}(r)))| . 
\end{align*} 
Note that $|E(T(H_{[n]}(r))) | = N_r([n])$. Hence, \eqref{eq:apnr} implies, $R_{1,n} \rightarrow \frac{\lambda}{r-1}$.  For $R_{2, n}$ note that $|\cK(t, T(H_{[n]}(r))) |$ counts the number of pairs of $r$-APs that have $t$ elements in common, for $t \in [2, r-1]$. Hence, having chosen the first $r$-AP in $O(n^2)$ ways, there are only $O(1)$ ways to select the second $r$-AP with $t$ elements in common with the first, for $t \in [2, r-1]$. Therefore, $$\frac{|\cK(t, T(H_{[n]}(r)))|}{c_n^{2r-t-1}} = \left(\frac{n^2}{c_n^{2r-t-1}} \right) \rightarrow 0, $$
for $t \in [2, r-1]$. Hence, $R_{2, n} \rightarrow 0$, and $\Var T(H_{[n]}(r) \rightarrow \lambda$. Theorem \ref{thm:THn} then implies, $W_r([n]) \dto \Pois(\frac{\lambda}{r-1})$, completing the proof of Corollary \ref{cor:apn}.

\subsection{Proof of Corollary \ref{cor:randomhypergraph} }
\label{sec:randomhypergraphpf}

The result in Theorem \ref{thm:hypergraphr} can be
 extended to random hypergraphs by the following lemma when the limits in \eqref{eq:meanvariancer} hold in probability.

\begin{lem}\label{lm:randomhypergraph}
    Let $\bm H_n = (H_n^{(1)}, H_n^{(2)})$ be a sequence of random 2-hypermultiplexes, where both $H_n^{(1)}$ and $H_n^{(2)}$ are $r$-uniform, which is independent of the coloring distribution such that the following hold: 
    \begin{align}\label{eq:meanvarianceprobability} 
   \bE [ \bm T(\bm H_n) | \bm H_n ] \pto \begin{pmatrix}
  \lambda_1 \\ 
  \lambda_2
  \end{pmatrix} \quad \text{ and } \quad \mathrm{Var}[ \bm T(\bm H_n) | \bm H_n ] \pto \begin{pmatrix}
  \lambda_1 & \lambda_{1, 2} \\ 
  \lambda_{1, 2} & \lambda_2
  \end{pmatrix} , 
    \end{align}
   for constants $\lambda_1, \lambda_2, \lambda_{1,2} \geq 0$, with $\lambda_{1,2} \le \lambda_1$ and $\lambda_{1,2} \le \lambda_2$. 
Then
    $$ \bm T(\bm H_n) = \left(\begin{array}{l}
        T(H_n^{(1)}) \\
        T(H_n^{(2)})
        \end{array}\right) \dto \left(\begin{array}{l}
            Z_1 + Z_{1,2}\\
           Z_2 + Z_{1,2}
            \end{array}\right) : = \bm Z , $$
    where $Z_1 \sim \Pois(\lambda_1 - \lambda_{1,2}), Z_2 \sim \Pois(\lambda_2 - \lambda_{1,2}), Z_{1,2} \sim \Pois(\lambda_{1,2})$ are independent.
\end{lem}

\begin{proof}
    The assumption \eqref{eq:meanvarianceprobability} implies the existence of positive reals $\varepsilon_n \rightarrow 0$ such that $\lim_{n\rightarrow\infty}\bP(A_n)=0$, where 
\begin{align*}
    A_n :=\{ \bm{H}_n := (H_n^{(1)}, H_n^{(2)}): \max\{& |\bE[T(H_n^{(i)})\mid \bm{H}_n] - \lambda_{i} |,|\Var[T(H_n^{(i)})\mid \bm{H}_n] - \lambda_{i} |, \\
        & \quad |\Cov [ T(H_n^{(1)}), T(H_n^{(2)})\mid \bm{H}_n ] -\lambda_{1,2}|\} > \varepsilon_n \text{ for } i = 1,2\}.
\end{align*} 
Thus, given any function $h:\{\bZ_{+} \cup \{0 \}\}^2 \mapsto [0,1]$ 
\begin{align*}
|\bE h( \bm T(\bm{H}_n) )-\bE h( \bm Z )|
\le \bP(A_n) +\sup_{ \bm{H}_n \in A_n^c} |\bE [ h( \bm T(\bm{H}_n ) )| \bm{H}_n ] - \bE [ h( \bm Z) ] |.
\end{align*}  

It thus suffices to prove that the second term on the right-hand side above converges to $0$. If not, there exists a deterministic sequence of 2-hypermultiplexes $(H_n^{(1)}, H_n^{(2)})$ such that $$\bE T(H_n^{(1)}) \rightarrow \lambda_1, \Var T(H_n^{(1)}) \rightarrow \lambda_{1}, \bE T(H_n^{(2)}) \rightarrow \lambda_2,  \Var T(H_n^{(2)}) \rightarrow \lambda_{2}, \Cov[T(H_n^{(1)}), T(H_n^{(2)})] \rightarrow \lambda_{1,2},$$ but
$$\left(\begin{array}{c}
    T(H_n^{(1)}) \\
    T(H_n^{(2)})
    \end{array}\right) \not \dto \left(\begin{array}{c}
        Z_1 + Z_{1,2}\\
       Z_2 + Z_{1,2}
        \end{array}\right) , $$ 
where $Z_1 \sim \Pois(\lambda_1 - \lambda_{1,2}), Z_2 \sim \Pois(\lambda_2 - \lambda_{1,2}), Z_{1,2} \sim \Pois(\lambda_{1,2})$ are independent, which is a contradiction to Theorem \ref{thm:hypergraphr}. 
\end{proof}

Using the above result we can complete the proof of Corollary \ref{cor:randomhypergraph}. Suppose $\bm H_n = (H_n^{(1)}, H_n^{(2)} ) \sim \bm{H}_r(n, p, \rho)$. From the law large numbers and \eqref{eq:p12}, 
\begin{align*}  
   \bE [ \bm T(\bm H_n) | \bm H_n ] = \begin{pmatrix}
  \frac{|E(H_n^{(1)})|}{c_n^{r-1}} \\ 
  \frac{|E(H_n^{(2)})|}{c_n^{r-1}}
  \end{pmatrix} = (1+o_P(1)) \begin{pmatrix}
  \frac{ \bE[|E(H_n^{(1)})|]}{c_n^{r-1}} \\ 
  \frac{ \bE[|E(H_n^{(2)})|]}{c_n^{r-1}}
  \end{pmatrix} 
   \pto \begin{pmatrix}
  \lambda \\ 
  \lambda
  \end{pmatrix} .
\end{align*}  
Also, for $i \in \{1, 2\}$, from \eqref{eq:varianceTHn} we have, 
$\Var[ T(H_n^{(i)})| \bm{H}_n ] = R_{1,n} + R_{2,n}$,  
where 
\begin{align*}
R_{1,n}^{(i)} = \frac{1}{c_n^{r-1}}\left(1 - \frac{1}{c_n^{r-1}}\right)|E(H_n^{(i)})| \pto \lambda , 
\end{align*} 
and 
$$
R_{2,n}^{(i)} = \sum_{t=2}^{r-1}\frac{1}{c_n^{2r-t-1}}\left(1 - \frac{1}{c_n^{t-1}}\right)|\cK(t, H_n^{(i)})| . $$
Note that $\bE[|\cK(t, H_n^{(i)})|] = \Theta(n^{2r-t} p^2)$. Hence, using the fact that $c_n = \Theta( n^{\frac{r}{r-1}} p^{\frac{1}{r-1}})$ from \eqref{eq:p12}, we have 
\begin{align}\label{eq:RH}
\frac{\bE[|\cK(t, H_n^{(i)})|]}{c_n^{2r-t-1}} = \Theta\left(\frac{n^{\frac{r(t-1)}{r-1}} p^{\frac{t-1}{r-1}} }{n^t} \right) = \Theta\left(\frac{ p^{\frac{t-1}{r-1}} }{n^{\frac{r-t}{r-1}}} \right) \rightarrow 0 , 
\end{align}
 for $t \in [2, r-1]$. This implies, $R_{2,n}^{(i)} \pto 0$, and therefore $\Var[ T(H_n^{(i)})| \bm{H}_n ] \pto \lambda$. 

Next, we look at the covariance term. As in \eqref{eq:H1H2Q}, we can write 
$$\Cov[(T(H_n^{(1)}), T(H_n^{(2)}))\mid \bm{H}_n] = \bE[Q_{1,n}\mid \bm{H}_n] + \bE[Q_{2,n}\mid \bm{H}_n] , $$
where 
$$Q_{1,n} = \frac{1}{c_n^{r-1}}\left(1 - \frac{1}{c_n^{r-1}}\right) |E(H_n^{(1, 2)})| \pto \lambda_{1, 2} , $$
by \eqref{eq:p12} and the law of large numbers, 
and 
$$Q_{2,n} = \sum_{t=2}^{r-1}\frac{1}{c_n^{2r-t-1}}\left(1 - \frac{1}{c_n^{t-1}}\right)|\cK(t, H_n^{(1)}, H_n^{(2)})|. $$ 
Note that $\bE[|\cK(t, H_n^{(1)}, H_n^{(2)})|] = \Theta(n^{2r-t} p_{1, 2})$. Hence, using the fact that $c_n = \Theta( n^{\frac{r}{r-1}} p_{1,2}^{\frac{1}{r-1}})$ from \eqref{eq:p12}, we have, as in \eqref{eq:RH},  
$$\frac{\bE[|\cK(t, H_n^{(1)}, H^{(2)} )|]}{c_n^{2r-t-1}} \rightarrow 0 , $$ for $t \in [2, r-1]$. This implies, $Q_{2,n} \pto 0$ and therefore $\Cov[(T(H_n^{(1)}), T(H_n^{(2)}))\mid \bm{H}_n] \pto \lambda_{1, 2}$. Hence, the conditions of Lemma \ref{lm:randomhypergraph} are satisfied and we have the result in Corollary \ref{cor:randomhypergraph}.  \hfill $\Box$

\subsection{Proof of Corollary \ref{cor:WHnlambda} }
\label{sec:WHnlambdapf}

Recalling the definition of $\mathsf{W}(H_n)$ from \eqref{eq:WHn} we have  
\begin{align}\label{eq:lambdaWHn}
\bE \mathsf{W}(H_n) =  \sum_{ \bm e \in {E(H_n)} }   \frac{ w_{\bm e} }{c_n^{r-1}} = \sum_{i=1}^K i \cdot \frac{ |E(H_n^{(i)})| }{c_n^{r-1}}  \rightarrow \lambda , 
\end{align}
by assumption \eqref{eq:lambdaRHn}. Next, note that 
$$\Var \mathsf{W}(H_n) = U_{1, n} + U_{2, n} , $$ 
where 
\begin{align}\label{eq:varianceR1}
U_{1, n} = \sum_{ \bm e \in {E(H_n)} }  w_{\bm e}^2 \frac{1}{c_n^{r-1}}\left(1 - \frac{1}{c_n^{r-1}}\right)  \geq \sum_{ \bm e \in {E(H_n)} } w_{\bm e} \frac{1}{c_n^{r-1}}\left(1 - \frac{1}{c_n^{r-1}}\right) \rightarrow \lambda , 
\end{align}
using \eqref{eq:lambdaWHn}, and 
\begin{align} 
U_{2, n} & = \sum_{t=2}^{r-1} \sum_{ \bm e_1, \bm e_2 \in E(H_n): |\bm e_1 \cap \bm e_2| = t }    w_{\bm e_1} w_{\bm e_2}   \frac{1}{c_n^{2r-t-1}}\left(1 - \frac{1}{c_n^{t-1}}\right) \nonumber \\ 
& \geq \sum_{t=2}^{r-1} \sum_{ \bm e_1, \bm e_2 \in E(H_n): |\bm e_1 \cap \bm e_2| = t }   \frac{1}{c_n^{2r-t-1}}\left(1 - \frac{1}{c_n^{t-1}}\right) \quad \text{(since $w(\bm e) \geq 1$, for $\bm e \in E(H_n)$)} \nonumber \\ 
& \geq \sum_{t=2}^{r-1} |\cK(t, H_n)| \frac{1}{c_n^{2r-t-1}}\left(1 - \frac{1}{c_n^{t-1}}\right) \geq 0 . 
\label{eq:covarianceR2} 
\end{align}
Since $\Var \mathsf{W}(H_n) \rightarrow \lambda$ by assumption \eqref{eq:lambdaRHn}, equalities must hold in \eqref{eq:varianceR1} and \eqref{eq:covarianceR2} in the limit. Specifically, we have $U_{2, n} \rightarrow 0$, which 
means $|\cK(t, H_n)| = o(c_n^{2r-t-1})$, for $t \in [2, r-1]$. Hence, 
\begin{align*} 
  |\cK(t, H_n^{(i)})| \leq  |\cK(t, H_n)|= o(c_n^{2r-t-1}) , 
\end{align*}
   for $t \in [2, r-1]$ and $1 \leq i \leq K$. Also, from \eqref{eq:varianceR1} we have $U_{1, n} \rightarrow \lambda$,  which means 
   $$U_{1, n} - \sum_{e \in E(H_n)} \frac{w_{\bm e}}{c_n^{r-1}}\left(1 - \frac{1}{c_n^{r-1}}\right) = \sum_{e \in E(H_n)} w_{\bm e}(w_{\bm e} - 1 ) \frac{1}{c_n^{r-1}}\left(1 - \frac{1}{c_n^{r-1}}\right) = \sum_{i=2}^{K} i(i-1) \frac{|E(H_n^{(i)})|}{c_n^{r-1}} \rightarrow 0.$$ 
   This implies, $\frac{|E(H_n^{(i)})|}{c_n^{r-1}} \rightarrow 0$, 
for $2 \leq i \leq K$. Hence, from \eqref{eq:lambdaWHn}, $\frac{|E(H_n^{(1)})|}{c_n^{r-1}} \rightarrow \lambda$. Therefore, applying Proposition \ref{ppn:dmultiplegraphs}, the result in Corollary \ref{cor:WHnlambda} follows. \hfill $\Box$

\small

\subsection*{Acknowledgements} B. B. Bhattacharya was supported by NSF CAREER grant DMS 2046393 and a Sloan Research Fellowship.

\bibliographystyle{plain}
\bibliography{bibfile}

\appendix 

\normalsize 

\section{A Hypergraph where the First Moment Phenomenon Does Not Hold}
\label{sec:hypergraphexample}

Let $H_n= (V(H_n), E(H_n))$ be a sequence of 3-uniform hypergraphs, with $V(H_n) = [n]$ and $E(H_n) = \{(1, a, b): 2 \leq a < b \leq n\}$, for $n \geq 3$. Clearly, $|E(H_n)| = {n-1 \choose 2} = \Theta(n^2)$. Hence, choosing $c_n = n$ ensures 
$$\bE T(H_n) = \frac{E(H_n)}{c_n^2} \rightarrow \frac{1}{2}.$$
However, in this case $T(H_n)$ does not converge in distribution to $\Pois(\frac{1}{2})$. To see this, observe that 
$$T(H_n) = \sum_{2 \leq a < b \leq n} \bm 1\{X_1 = X_a = X_b\}, $$
where $X_v$ is the color of the vertex $v \in [n]$. Conditioning on the event $\{X_1 = a\}$, for some $a \in [c_n]$, notice that  
$$T(H_n) | \{X_1 = a\} = {Y_a \choose 2}, $$ 
where $Y_a$ is the number of vertices in $[2, n]$ with color $a$. Clearly, $Y_a \sim \mathrm{Bin}(n-1, 1/n) \dto \Pois(1)$, for all $a \in [c_n]$. Hence, $$T(H_n) \dto {Z \choose 2}, $$
where $Z \sim \Pois(1)$. Clearly, $ {Z \choose 2}$ does not have a $\Pois(\frac{1}{2})$ distribution, in particular, it does not take the values $0$ or $1$. 
This shows that, unlike in the graph case (recall the discussion following \eqref{eq:ETHn}), the first moment phenomenon does not hold for general $r$-uniform hypergraphs, for $r \geq 3$. 

\section{Failure of the Second-Moment Phenomenon for $3$-Hypermultiplexes}
\label{sec:3layersH}

Consider two sequences of 3-multiplexes $\bm H_n = (H_n^{(1)}, H_n^{(2)}, H_n^{(3)})$ and $\tilde{\bm H}_n = (\tilde H_n^{(1)}, \tilde H_n^{(2)}, \tilde H_n^{(3)})$ defined as follows: 

\begin{itemize} 

\item Let $A_1, A_2, A_3$ be 3 sets of size $n$ such that 
$$\frac{|A_1\cap A_2 \cap A_3|}{n} \rightarrow \lambda \in (0, \tfrac{1}{4})$$ 
and 
$$A_1\cap A_2 \cap A_3 = A_1\cap A_2 = A_2\cap A_3 = A_1\cap A_3,$$
as shown in Figure \ref{fig:hypermultiplex example} (a). Let $H_n^{(1)}, H_n^{(2)}, H_n^{(3)}$ be the complete graphs with vertex sets $A_1, A_2, A_3$, respectively. Define $$\bm H_n = (H_n^{(1)}, H_n^{(2)}, H_n^{(3)}).$$

\item Let $\tilde A_1, \tilde A_2, \tilde A_3$ be 3 sets of size $n$ such that $\tilde A_1\cap \tilde A_2 \cap \tilde A_3 = \emptyset$, 
$$\frac{|\tilde A_1\cap \tilde A_2|}{n} = \frac{|\tilde A_2\cap \tilde A_3|}{n} = \frac{| \tilde A_1\cap \tilde A_3|}{n} \rightarrow \lambda   \in (0, \tfrac{1}{4}) ,$$ 
as shown in Figure \ref{fig:hypermultiplex example} (b). Let $\tilde H_n^{(1)}, \tilde H_n^{(2)}, \tilde H_n^{(3)}$ be the complete graphs with vertex sets $\tilde A_1, \tilde A_2, \tilde A_3$, respectively. Define $$\tilde{\bm H}_n = (\tilde H_n^{(1)}, \tilde H_n^{(2)}, \tilde H_n^{(3)}).$$

\end{itemize}

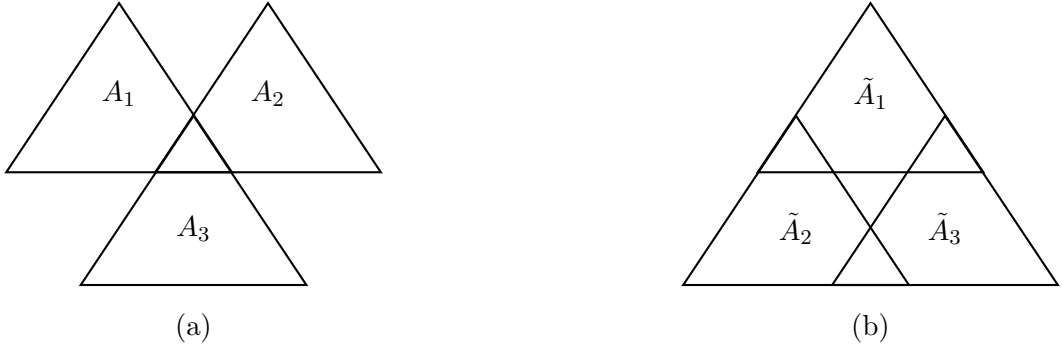
\begin{figure}[h]
    \centering
    \begin{tikzpicture}[scale=1.5]
    % First set of triangles (left)
    \draw[thick] (-0.16,2.5) -- (-1.16,1) -- (0.84,1) -- cycle;
    \draw[thick] (1.16,2.5) -- (0.16,1) -- (2.16,1) -- cycle;
    \draw[thick] (0.5,1.5) -- (-0.5,0) -- (1.5,0) -- cycle;
    
    % Labels for first set
    \node[thick] at (-0.16,1.7) {$A_1$};
    \node[thick] at (1.16,1.7) {$A_2$};
    \node[thick] at (0.5,0.5) {$A_3$};
    \node[thick] at (0.5,-0.375) {(a)}; 
    
    % Second set of triangles (right)
    \begin{scope}[xshift=6cm]
        \draw[thick] (-0.16,1.5) -- (-1.16,0) -- (0.84,0) -- cycle;
        \draw[thick] (1.16,1.5) -- (0.16,0) -- (2.16,0) -- cycle;
        \draw[thick] (0.5,2.5) -- (-0.5,1) -- (1.5,1) -- cycle;
        
        % Labels for second set
        \node[thick] at (0.5,1.7) {$\tilde A_1$};
        \node[thick] at (-0.16,0.5) {$\tilde A_2$};
        \node[thick] at (1.16,0.5) {$\tilde A_3$}; 
        \node[thick] at (0.5,-0.375) {(b)}; 

    \end{scope}
    \end{tikzpicture}
    \caption{\small{Illustration demonstrating the failure of the second-moment phenomenon in characterizing the limiting distribution of $3$-multiplexes. %The two sequences of $3$-hypermultiplexes, $\bm H_n = (H_n^{(1)}, H_n^{(2)}, H_n^{(3)})$ and $\tilde{\bm H}_n = (\tilde H_n^{(1)}, \tilde H_n^{(2)}, \tilde H_n^{(3)})$, share the same first and second moments (that is, identical mean vectors and covariance matrices). 
    }}
    \label{fig:hypermultiplex example}
\end{figure}

Now, choose $c_n= n^{2}$. Then, since $\frac{{n \choose 2}}{c_n} \rightarrow \frac{1}{2}$, 
  \begin{align*}
     \bE [ \bm T(\bm H_n) ] =  \bE [ \bm T(\tilde{\bm H}_n) ]  \rightarrow \begin{pmatrix}
      \frac{1}{2} \\ 
      \frac{1}{2} \\
      \frac{1}{2}
      \end{pmatrix} .
      \end{align*}
Also, a direct calculation shows that 
    \begin{align*}
        \lim_{n \rightarrow \infty} \mathrm{Var}[ \bm T(\bm H_n) ] &= \lim_{n \rightarrow \infty} \mathrm{Var}[ \bm T(\tilde{\bm H}_n) ] = \begin{pmatrix}
      \frac{1}{2} & \lambda' & \lambda' \\ 
      \lambda' & \frac{1}{2} &  \lambda' \\
      \lambda' & \lambda' & \frac{1}{2}
      \end{pmatrix}. 
    \end{align*}
where $\lambda' = \frac{1}{2}\lambda^2$.
%\xx{The off-diagonals should be $\frac{1}{2}\lambda^2$ because $\lambda n$ is the number of nodes in the intersection, so the number of edges in this induced complete subgraph will become
%$$
%\binom{\lambda n}{2}
%$$
%}
  This shows that the first and second moments (mean vectors and covariance matrices) of $\bm T(\bm H_n)$ and $\bm T(\tilde{\bm H}_n)$ asymptotically equal. However, it turns out that the limiting distributions of $\bm T(\bm H_n)$ and $\bm T(\tilde{\bm H}_n)$ are not the same. To see this, define $H_n^{(*)} = H_n^{(1)} \cap H_n^{(2)} \cap H_n^{(3)}$, that is, the graph with edge set $E(H_n^{(1)}) \cap E(H_n^{(2)}) \cap E(H_n^{(3)})$. Then applying Proposition \ref{ppn:dmultiplegraphs}, we obtain the following: %with $H_n^{(*)} = H_n^{(1)} \cap H_n^{(2)} \cap H_n^{(3)}$, or the hypergraph with vertex set $V_n$ and edge set $E(H_n^{(1)}) \cap E(H_n^{(2)}) \cap E(H_n^{(3)})$,\footnote{More details on the notation $H_n^{(1)}\setminus(H_n^{(2}), H_n^{(1)}\cap(H_n^{(2})$ and $H_n^{(1)}\cup(H_n^{(2})$ are explained in Section \ref{sec:hypergraphrpf}}
$$\left(\begin{array}{c}
    T(H_n^{(1)}\setminus H_n^{(*)}) \\ 
    T(H_n^{(2)}\setminus H_n^{(*)}) \\
    T(H_n^{(3)}\setminus H_n^{(*)}) \\
    H_n^{(*)}
    \end{array}\right) \dto \left(\begin{array}{c}
        Z_1 \\ 
        Z_2 \\
        Z_3 \\
        Z_{1,2,3}
        \end{array}\right), $$
        where $Z_1, Z_2, Z_3\sim \Pois(\frac{1}{2} - \lambda'), Z_{1,2,3} \sim \Pois(\lambda')$ are independent. This implies, 
\begin{align}\label{eq:HA}
\left(\begin{array}{c}
    T(H_n^{(1)}) \\ 
    T(H_n^{(2)}) \\
    T(H_n^{(3)}) \\
    \end{array}\right) \dto \left(\begin{array}{c}
        Z_1 + Z_{1,2,3}\\ 
        Z_2 + Z_{1,2,3}\\
        Z_3 + Z_{1,2,3}\\
    \end{array}\right). 
    \end{align}
Similarly, it can be shown that 
\begin{align}\label{eq:HApairwise}
\left(\begin{array}{c}
    T(\tilde H_n^{(1)}) \\ 
    T(\tilde H_n^{(2)}) \\
    T(\tilde H_n^{(3)}) \\
    \end{array}\right) \dto \left(\begin{array}{c}
        Z_1' + Z_{1,2} + Z_{1,3}\\ 
        Z_2' + Z_{1,2} + Z_{2,3}\\
        Z_3' + Z_{1,3} + Z_{2,3}\\
    \end{array}\right), 
    \end{align}
    where $Z_1', Z_2', Z_3'\sim \Pois(\frac{1}{2} - 2\lambda'), Z_{1,2}, Z_{1,3}, Z_{2,3} \sim \Pois(\lambda')$ are independent.
Clearly, the limiting distributions in \eqref{eq:HA} and \eqref{eq:HApairwise} are not the same.

This example show that, although the first 2 moments of $\bm T(\bm H_n)$ and $T(\tilde{\bm H}_n)$ converge to the same limit, their limiting distributions are not the same. Hence, the second moment phenomenon can fail for (hyper)multiplexes with 3 or more layers.

%The following example shows that in this case only assuming the convergence of the mean vector and the covariance matrix is not enough to ensure distributional convergence. In other words, for more than two layers, the second moment phenomenon does not hold in general. 

\section{Ordering Lemma}
\label{sec:connectedpf}

In this section we prove the ordering lemma used in the proof of Lemma \ref{lm:countS}. The proof is essentially a verbatim adaption of \cite[Lemma A.1]{bhattacharya2020second} to the hypergraph setting.

\begin{lem}\label{lm:connected} Fix $k \geq 1$ and suppose $S = (\bm e_1, \bm e_2, \ldots, \bm e_k)$ be a collection of hyperedges (not necessarily distinct) of size $r \geq 2$ such that the hypergraph formed the union of the edges in $S$ is connected and $|\bigcup_{j=1}^k \bm e_j| <  b r-b+1$, where $b$ is the number of distinct hyperedges in $S$. Then there exists an ordering (permutation) $\sigma: [k] \rightarrow [k] $ such that the following hold: 
\begin{itemize}
\item for every $i \in [2, k]$, $t_i(\sigma)  \geq 1$, and  

\item for some $i \in [2, k]$, $t_i(\sigma) \in [2, r-1] $, 
\end{itemize} 
where 
$$t_i(\sigma) = \left|\bm e_{\sigma(i)} \bigcap \left(\bigcup_{j = 1}^{i-1}\bm e_{\sigma(j)}\right) \right| , $$ 
for $i \in [2, r]$. 
\end{lem}

\begin{proof}
Since $S$ is connected, there exists an ordering $\sigma$, such that $t_i(\sigma) \geq 1$ for every $i \in [2,k]$. Suppose, towards a contradiction, $t_i(\sigma) \in \{1, r\}$, for every $i \in [2,k]$. This means there can be 2 cases:
\begin{itemize}

\item[{\it Case} 1:] For every $i \in [2, k]$, either $t_i(\sigma) = 1$, or $\bm e_{\sigma(i)} \in \{\bm e_{\sigma(1)},\dots,\bm e_{\sigma(i-1)}\}$.
$$\ell=\big|\{i \in[2,k]: t_i(\sigma) = 1\}\big| .$$ 
Then, $b = 1+\ell$ and $|\bigcup_{j=1}^k \bm e_j| = r + \ell (r-1)$. This yields a contradiction, because $$\Big|\bigcup_{j=1}^k \bm e_j\Big| = r + (b-1)(r-1) = br-b+1 .$$

\item[{\it Case} 2:]  There exists $i \in [2, k]$ such that $t_i(\sigma)=r$ and $\bm e_{\sigma(i)} \notin \{\bm e_{\sigma(1)},\dots,\bm e_{\sigma(i-1)}\}$. Define 
\begin{equation*} 
i_0 = \inf \big\{i \in [2, k] : t_i(\sigma)=r~\textrm{and}~\bm e_{\sigma(i)} \notin \{\bm e_{\sigma(1)},\dots,\bm e_{\sigma(i-1)}\} \big\} 
\end{equation*} 
and 
\begin{equation*} 
i_1 = \inf \big\{1\leq i < i_0:\bm e_{\sigma(i_0)}\cap\bm e_{\sigma(i_1)} \neq \emptyset \} .
\end{equation*} 
Now, consider the following cases: 
\begin{itemize} 

\item[--] $\big|\bm e_{\sigma(i_0)}\cap\bm e_{\sigma(i_1)}\big| \geq 2$: In this case, consider a permutation $\tau: [k]\rightarrow [k]$ with $\tau(1) = \sigma(i_0)$, $\tau(2) = \sigma(i_1)$, and $t(i,\tau) \geq 1$,  for every $i \in [2,k]$. By the definition of $i_0$, it follows that $\bm e_{\sigma(i_1)} \neq\bm e_{\sigma(i_0)}$, and since $\big|\bm e_{\sigma(i_0)}\cap\bm e_{\sigma(i_1)}\big| \geq 2$, we have $t_2(\tau) \in [2, r-1]$, as required.

\item[--] $\big|\bm e_{\sigma(i_0)}\cap\bm e_{\sigma(i_1)}\big| = 1$: Let $\{s\} =\bm e_{\sigma(i_0)}\cap\bm e_{\sigma(i_1)}$. Define 
\begin{equation*} 
i_2 = \inf \big\{i_1< i < i_0: (\bm e_{\sigma(i_0)}\setminus \{s\})\cap\bm e_{\sigma(i)} \neq \emptyset \} . 
\end{equation*}	
Once again, there exists a permutation $\kappa: [k]\rightarrow [k]$ such that $\kappa(1) = \sigma(i_0)$, $\kappa(2) = \sigma(i_2)$ and $t_i(\kappa) \geq 1$, for every $i \in [2, k]$. So, if $\big|\bm e_{\sigma(i_0)}\cap\bm e_{\sigma(i_2)}\big| \geq 2$, then $t_2(\kappa) \in [2, r-1]$, as desired. Hence, suppose $\big|\bm e_{\sigma(i_0)}\cap\bm e_{\sigma(i_2)}\big| = 1$. Now, there exists a permutation $\theta: [k] \rightarrow [k]$ satisfying: 
	\[   
	\theta(i) = 
	\begin{cases}
	\sigma(i) &\quad\text{if} ~1\leq i \leq i_2 , \\
	\sigma(i_0) &\quad\text{if} ~ i=i_2+1 , \\ 
	\end{cases}
	\]
and $t_i(\theta) \geq 1$, for every $i \in [2,k]$.	Now, it is easy to see that $t_{i_2+1}(\theta) = 2$, completing the proof of Lemma \ref{lm:connected}. 
\end{itemize}
\end{itemize} 
\end{proof}

\end{document}